%% file: main.tex
\documentclass{article}

\input{packages}
\input{commands}

\author{Leon Bungert\thanks{Hausdorff Center for Mathematics, University of Bonn, Endenicher Allee 62, Villa Maria,
53115 Bonn, Germany. Email: \href{mailto:leon.bungert@hcm.uni-bonn.de}{leon.bungert@hcm.uni-bonn.de}} \and Philipp Wacker\thanks{Friedrich-Alexander University Erlangen-Nürnberg, Department Mathematics, Cauerstr. 11, 91058 Erlangen, Germany. Email: \href{mailto:wacker@math.fau.de}{wacker@math.fau.de}}}

\title{Complete Deterministic Dynamics and Spectral Decomposition of the Linear Ensemble Kalman Inversion}

\makeatletter
\let\blx@rerun@biber\relax
\makeatother

\begin{document}

\maketitle

\begin{abstract}
    The ensemble Kalman inversion (EKI) for the solution of Bayesian inverse problems of type {$y = A u +\eps$, with $u$ being an unknown parameter, $y$ a given datum, and $\eps$ measurement noise}, is a powerful tool usually derived from a sequential Monte Carlo point of view. 
    It describes the dynamics of an ensemble of particles $\{u^j(t)\}_{j=1}^J$, whose initial empirical measure is sampled from the prior, evolving over an artificial time $t$ towards an approximate solution of the inverse problem{, with $t=1$ emulating the posterior, and $t\to\infty$ corresponding to the under-regularized minimum-norm solution of the inverse problem}.
    Using spectral techniques, we provide a complete description of the deterministic dynamics of EKI and its asymptotic behavior in parameter space.
    In particular, we analyze the dynamics of {naive} EKI and mean-field EKI {with a special focus on their time asymptotic behavior.}
    Furthermore, we show that---even in the deterministic case---residuals in parameter space do not decrease monotonously in the Euclidean norm and suggest a problem-adapted norm, where monotonicity can be proved. 
    Finally, we derive a system of ordinary differential equations governing the spectrum and eigenvectors of the covariance matrix.
    {While the analysis is aimed at the EKI, we believe that it can be applied to understand more general particle-based dynamical systems. }
\end{abstract}

\tableofcontents

\section{Introduction}

In this article we study the ensemble Kalman approach for solving {a linear} inverse problem of the form
\begin{align}\label{eq:inverse_prob}
    y = A {u} + \eps ,\quad \eps \sim \mathcal{N}(0,\Gamma),
\end{align}
where {$u \in \R^n$} is the unknown parameter of interest, $y\in \R^m$ are noisy measurements, $A\in \R^{m\times n}$ is a forward operator mapping the parameter space into the observation space, and $\Gamma \in \R^{m\times m}$ is a covariance matrix of the noise model in the measurement process which gives $y$. Following the Bayesian approach to inverse problems \cite{kaipio2006statistical}, we specify a prior measure $\mu_0$ on the set of feasible parameters $u$. Bayes' theorem then shows a way of incorporating the data $y$ into the prior, yielding a posterior measure $\mu$ such that 
\begin{align*}
    \frac{\d\mu}{\d\mu_0}(u) \propto \exp\left(-\frac{1}{2}\|y-Au\|_\Gamma^2\right)
\end{align*}
Here (and in the following) we define the weighted inner product and norm as $\langle x,y\rangle_H:=\langle x,H^{-1}y\rangle$ and $\norm{x}_H^2 :=\langle x, H^{-1} x\rangle$ for a symmetric positive definite matrix~$H$.

We continue with setting up the notation, a historical discussion and an introduction into the main techniques of the paper. {An overview of the} novel contributions of this manuscript can be found on page \pageref{sec:contribution}

\paragraph{{The ensemble Kalman methodology applied to inverse problems}}

{The ensemble Kalman method was originally generalized from a method for linear Gaussian state estimation \cite{kalman1960new} and derived in the field of data assimilation, for the task of state estimation \cite{evensen1994sequential,evensen2009data}.
Quickly it was adopted as a method for inversion by \cite{chen2012ensemble,emerick2013ensemble} in the field of oil reservoir modelling, taking inspiration from sequential Monte Carlo \cite{del2006sequential}. In all these works, the prior is (at least approximately) transported to the posterior in one (artificial) time unit $t=1$, {we thus label such methods as ``ensemble sequential Monte Carlo (EnSMC)''}. It was soon framed as a derivative-free optimization method in \cite{reich2011dynamical}, which kept with the idea of transporting the prior to the posterior in one time unit $t=1$, but already hinted at the possibility of studying the behaviour for $t\to\infty$.

This {led} to the advent of the ensemble Kalman methodology for inverse problems, commonly abbreviated EKI, starting with \cite{iglesias2013ensemble}, followed by detailed continuum limit analysis in various senses (continuous time, mean-field limit etc.) in, e.g., \cite{schillings2017analysis,schillings2018convergence,blomker2019well,herty2019kinetic}, extended to include Tikhonov regularization in \cite{chada2020tikhonov}, and generalized to the ensemble Kalman Sampler (EKS) \cite{garbuno2020interacting} and the closely related ALDI method in \cite{nusken2019note,garbuno2020affine}. {Here, a divergence of the philosophy of the various members of the Ensemble Kalman family can be observed: EnSMC methods like EKS and ALDI approximate the posterior measure by sampling. In contrast, EKI, TEKI and {derivatives of these methods} try to solve an optimization problem by using the ensemble's pointwise evaluations for implicit gradient approximations. In particular, EKI and TEKI are expected to contract around an optimal point, with their ensemble covariance collapsing. On the other hand, EKS and ALDI will try to match the posterior's covariance structure.} 

{In a recent survey \citep{calvello2022ensemble}, various versions of the Ensemble Kalman method were unified and anlysed, in particular from the mean-field perspective.}
It should be noted that many techniques employed in the analysis of these various methods were pioneered and influenced by research done on the ensemble Kalman filter in data assimilation, see \cite{bergemann2012ensemble} for continuous-time analysis, and works like \cite{reich2011dynamical,majda2012filtering,law2015data,reich2015probabilistic,de2018long}. Due to different communities' preferred notations, the nomenclature around ensemble Kalman $\{$filtering, inversion$\}$ is not always clearly distinguished and has considerable overlap, with similar techniques bearing different names in different fields (and the other way around).  }


The basic idea {of the ensemble Kalman method} is to replace all measures involved with an empirical measure generated by an ensemble of particles: An initial ensemble of particles $\{u_0^j\}_{j=1}^J$ (usually sampled from the prior $\mu_0$) is considered a surrogate for the prior, and transitioning from prior to posterior amounts to moving the ensemble to new positions $u_1^j$, with the posterior ensemble $\{u_1^j\}_{j=1}^J$ standing in for the posterior. It can be shown that this particle update is given by the ensemble Kalman inversion (EKI)
\begin{align}\label{eq:enk_it_1step}
    u^j_{1} = u^j_0 -  C(u_0) A^T(A C(u_0) A^T + \Gamma)^{-1}(Au^j_0 - y),
\end{align}
where $C(u_0)$ is the sample covariance of the initial ensemble $\{u_0^j\}_{j=1}^J$.

Under linear and Gaussian assumptions, and for an initial ensemble with {sample} covariance matching exactly the prior covariance {(i.e., in particular $J\geq n$)}, it can be shown that the empirical measure given by the final particles $\{u_1^j\}_{j=1}^J$ is exactly the Gaussian measure identical to the posterior. Although only exact in this linear and Gaussian regime (see \cite{ernst2015analysis}), the computational benefit of replacing measures by empirical measures has prompted usage of this methodology in a more broader context, for example for nonlinear observation operators and non-Gaussian priors. {Also, in the underdetermined case of $J<n$ which is especially relevant in practical applications, the subspace property {(stated explicitly in \citep{iglesias2013ensemble} but well-known in the community before then) }of the Ensemble Kalman method leads to regularization by dimensionality reduction, which is one of the reasons of good performance of the EKI method.}

If prior and posterior are very different from each other, then the particle updates \labelcref{eq:enk_it_1step} exhibit a large jump. Exchanging this one-step algorithm by a many-step iteration is supposed to yield a much smoother transition between prior and posterior, alongside a promise of computational improvements stemming from ``easing into'' the posterior. 
By introducing intermediate time steps (and thus intermediate measures $\mu_k$ ``interpolating'' between the prior $\mu_0$ and the posterior $\mu$), one obtains the iteration
\begin{align}\label{eq:enk_it}
    u^j_{n+1} = u^j_n -  C(u_n) A^T(A C(u_n) A^T + \tau^{-1}\Gamma)^{-1}(Au^j_n -  {y_n^j}), \quad n=0,\ldots,N,
\end{align}
{where $y_n^j \sim \mathcal N(y, \tau^{-1}\Sigma)$ are i.i.d. perturbed observations, $C(u_n)$ is a notational shorthand for the sample covariance $C(\{u_n^j\}_j)$,} and $\tau=1/N$ with $N\in\N$ is a time step.  If $\Sigma = \Gamma$, this is the stochastic iterative EKI dynamics \cite{garbuno2020interacting}. 
If we set $\Sigma = 0$, we obtain the {naive} iterative EKI dynamics, {and if additionally $N=1$ \labelcref{eq:enk_it} coincides with the ``vanilla'' EKI \labelcref{eq:enk_it_1step}}.

The data $y$ is used in a perturbed form $y_n^j$ where the perturbation is done in such a way as to match the statistical properties of the observation noise. The idea is that perturbing the data additionally would help the dynamics explore the space state better. Also, we will see that without this perturbation, in contrast to~\labelcref{eq:enk_it_1step}, the continuous EKI dynamics does not recover the posterior.

After some algebra it can be seen that \labelcref{eq:enk_it} can be equivalently reformulated as solution of a variational regularization problem:
\begin{align}\label{eq:enk_it_optim}
     u^j_{n+1} = \argmin_{u} \frac{\tau}{2}\norm{Au -  {y_n^j}}_\Gamma^2 + \frac{1}{2}\norm{u - u_n^j}_{ C(u_n)}^2.
\end{align}
{Note that there is a link here to the methods of randomized likelihood estimation \citep{chen2012ensemble} and Randomize-then-optimize \citep{bardsley2014randomize}, which also construct surrogates to the posterior measure by solving an optimization problem.}
 Furthermore, by linearity of \labelcref{eq:enk_it} the sample means $m_n:=\frac{1}{J}\sum_{j=1}^Ju_n^j$ also satisfy
\begin{align}\label{eq:enk_it_means_optim}
     m_{n+1} = \argmin_{u} \frac{\tau}{2}\norm{A u - y_n}_\Gamma^2 + \frac{1}{2}\norm{u - m_n}_{ C(u_n)}^2,
\end{align}
{where $y_n := \frac{1}{J}\sum_{j=1}^J y_n^j$ is the sample mean of the perturbed observations at iteration~$n$.}
Hence the evolution of the whole ensemble and the sample means can be seen as {stochastic} minimizing movement discretization of the gradient flow of the functional $u\mapsto \frac12\norm{Au-{y}}^2_\Gamma$ with respect to the varying norm $\norm{\cdot}_{ C(u_n)}^2$, cf. the discussion in \cite{armbruster2020stabilization,weissmann2020particle}, and see \cite{garbuno2020interacting} for a mean-field version of this structure.

\paragraph{Continuous EKI}

Following \cite{iglesias2013ensemble,schillings2017analysis}, sending $\tau\to 0$ one arrives at the continuous ensemble Kalman inversion method
\begin{equation} \label{eq:enkf_stoch}
     \dot{u}^j(t) = - C(t) A^T\Gamma^{-1}(Au^j(t)-y) +  C(t)A^T\Gamma^{-1}\sqrt{\Sigma}\dot{\wiener}^j(t)
\end{equation}
for $j\in\{1,\ldots, J\}$, where $t\mapsto \wiener^j(t)$ are i.i.d. Wiener processes and $\Sigma$ is a symmetric positive definite matrix, with interesting special cases being $\Sigma=0$ ({naive} EKI) and $\Sigma=\Gamma$ (stochastic EKI). {Both versions of \labelcref{eq:enkf_stoch} are particle dynamics aimed at optimization tasks. Indeed we will demonstrate that, although the one-step variant of EKI in \labelcref{eq:enk_it_1step} recovers the posterior, this property is lost via the discretization and continuum limit leading to \labelcref{eq:enkf_stoch}. If one is interested in sampling, the Ensemble Kalman Sampler \cite{garbuno2020interacting} is a better choice.}

Note that by construction $t=0$ corresponds to the prior and $t=1$ is supposed to approximate the posterior (although in which sense has to be specified). For time $t > 1$, the influence of the prior decreases further.
The quantities,
\begin{align}
    \label{eq:sample_cov}
     C(t) &:=\frac{1}{J}\sum_{j=1}^J(u^j(t)-m(t)) \otimes (u^j(t)-m(t)),\\
    \label{eq:sample_mean}
    m(t) &:= \frac{1}{J}\sum_{j=1}^J u^j(t),
\end{align}
denote the \emph{{sample} covariance} and \emph{{sample} mean} of the particles $u^j(t)$ at time $t>0$.

Apart from time-continuous limits of the ensemble Kalman inversion \labelcref{eq:enk_it}, another interesting limit is the \emph{mean-field} limit as the number of particles $J\in\N$ in the ensemble goes to infinity.
In this case, it is well known that the empirical measure $\rho_J(t) := \frac{1}{J}\sum_{j=1}^J \delta_{u^j(t)}$, where the particles $u^j(t)$ solve \labelcref{eq:enkf_stoch}, converge to a time-dependent measure $\rho(t,u)$ which solves a Fokker--Planck equation, see, e.g., \cite{law2016deterministic,ding2019mean,herty2019kinetic}.
{Letting $D^2\rho$ denotes the Hessian matrix of $\rho$ with respect to the variable $u$ and defining the covariance matrix and mean associated to such measures as
\begin{align}
    \mathfrak C(t) &:= \int (u - \mathfrak m(t))\otimes(u - \mathfrak m(t)) \d\rho(t,u),\label{eq:meanfield_cov}\\
    \mathfrak m(t) &:= \int u \d\rho(t,u), \label{eq:meanfield_mean}
\end{align}
the Fokker--Planck equation of mean-field EKI takes the form
\begin{align}\label{eq:mean-field-limit}
     \partial_t \rho = \div\left(\rho~\mathfrak C(t)A^T\Gamma^{-1}(A u - y)\right) + \frac12 \operatorname{Tr}(D^2\rho~\mathfrak C(t)A^T\Gamma^{-1}A\mathfrak C(t)).
\end{align}}
{Before continuing with a detailed exposition of set-up and notation, we start with a description of this manuscript's goals.}
\paragraph{Contributions of this article}\label{sec:contribution}
{
\begin{itemize}
    \item \cref{thm:cov_dynamics,thm:asymptotic_profile,thm:deterministic_EKI,thm:mean_field}: An explicit characterization of the  {\textbf{dynamics of {naive} and mean-field EKI}}, together with a detailed convergence analysis of these quantities for $t\to\infty$ and explicit rates of convergence.
    The main idea here is that the diagonalization of a suitable reweighting of the time-dependent covariance matrix remains constant in time which allows for a complete analysis of the dynamics of the different EKI variants.
    \item \cref{ex:nonmon}: A negative answer to the question \textbf{whether the residual spread of EKI decreases mono\-tonous\-ly}. 
    This quantity measures the average squared deviation of the ensemble to a specific solution of the inverse problem.
    This is even true for noise-free data and an optimal choice of solution.
    \item \cref{thm:eigenvectors}: A \textbf{spectral analysis of the covariance} $\cov(t)$ itself, leading to a coupled differential algebraic system characterizing the evolution of eigenvalues and eigenvectors of the ensemble covariance.
    \item \labelcref{eq:ODE_cov_emp,eq:ODE_mean_emp,eq:ODE_cov_mean_field,eq:ODE_mean_mean_field,eq:ODE_cov_av_emp,eq:ODE_mean_av_emp}:
    We show that \textbf{only mean-field EKI recovers the posterior} in unit time.
    This property is at the heart Kalman update formula \labelcref{eq:enk_it_1step} in the linear and Gaussian regime but gets lost by further discretization and in the continuum limit.
    {Although the EKI algorithm in \citep{iglesias2013ensemble} was only presented as an optimization method, clearly, variants that address the sampling problem are of interest, e.g., \citep{huang2022efficient}.}
\end{itemize}}

\paragraph{Various notions of mean and covariance}

In the context of different EKI methodologies it is very important to exactly specify what is meant by ``mean'' or ``covariance'': For given paths of Brownian motion, the ensemble given by \labelcref{eq:enkf_stoch} has an empirical (sample) covariance $ C$ \labelcref{eq:sample_cov} and an empirical (sample) mean $m$ \labelcref{eq:sample_mean}, i.e., the center of mass and a measure for the spread given by the particles. 
In addition, if we take the expectation $\E^\wiener$ with respect to the Wiener measure governing the Wiener noise terms $\wiener^j$, we get a different notion of mean and covariance, corresponding to the \emph{average {sample} mean} $\mathbf{m}=\E^\wiener m$  and \emph{average {sample} covariance} $\mathbf{C}=\E^\wiener C$. Finally, by considering the ``mean-field limit'' $J\to\infty$, one obtains a time-dependent measure $\rho_t$ with a mean and covariance (see \labelcref{eq:meanfield_mean,eq:meanfield_cov}).
Depending on the type of mean and covariance one is referring {to}, the governing equations for these two quantities change.

For instance, in the deterministic case of $\Sigma=0$ in \labelcref{eq:enkf_stoch}{, which we call ``naive EKI'',} the {sample} covariance and mean \labelcref{eq:sample_cov,eq:sample_mean} evolve according to (see~\cref{lem:derivation_ODEs})
\begin{align}
    \label{eq:ODE_cov_emp}
    \dot{ C}(t) &= -2 C(t) A^T \Gamma^{-1} A  C(t),\\
    \label{eq:ODE_mean_emp}
    \dot{m}(t) &= - C(t) A^T\Gamma^{-1}(A m(t)-y).
\end{align}
If $\Sigma=\Gamma$, these two empirical quantities evolve under a \emph{stochastic} differential equation, see \cref{lem:derivation_SDEs}.
In contrast, as outlined in \cite{garbuno2020interacting} and shown in \cref{lem:derivation_meanfield} the mean-field covariance and mean \labelcref{eq:meanfield_cov,eq:meanfield_mean} evolve according to the differential equation
\begin{align}
    \label{eq:ODE_cov_mean_field}
    \dot{\mathfrak C}(t) &= -\mathfrak C(t) A^T \Gamma^{-1} A \mathfrak C(t),\\
    \label{eq:ODE_mean_mean_field}
    \dot{\mathfrak m}(t) &= -\mathfrak C(t) A^T\Gamma^{-1}(A\mathfrak m(t)-y).
\end{align}
Finally, the average {sample} mean $\mathbf{m}(t):= \E^\wiener  m(t)$ and average {sample} covariance $\mathbf{C}(t) := \E^\wiener  C(t)$ of the EKI for $\Sigma = \Gamma$ evolve as (see \cref{lem:derivation_SDEs}):
\begin{align}
    \dot{\mathbf C}(t) &= -\frac{J+1}{J}\mathbf C(t) A^T \Gamma^{-1} A \mathbf C(t)\notag
    \\
    \label{eq:ODE_cov_av_emp}
    &\qquad 
    -
    \E^\wiener\left[(C(t)-\mathbf C(t))A^T\Gamma^{-1}A(C(t)-\mathbf C(t))\right],\\
    \dot{\mathbf m}(t) &= -\mathbf C(t) A^T\Gamma^{-1}(A\mathbf m(t)-y) 
    \notag
    \\
    \label{eq:ODE_mean_av_emp}
    &\qquad 
    -\E^\wiener\left[(C(t)-\mathbf C(t))A^T\Gamma^{-1}A(m(t)-\mathbf m(t))\right]
\end{align}
Note that {the first and the second as well as the leading term of the third covariance equation} only differ by constants.

{
\begin{remark}
The analysis presented in this manuscript can also straightforwardly be extended to the case of (linear) Tikhonov-regularized Ensemble Kalman inversion (TEKI) as developed in \cite{chada2020tikhonov}: Here, the data space is augmented by a copy of the parameter space in order to allow for additional Tikhonov regularization of the parameters in the form of a penalization term $\frac{1}{2}\|u\|_\Sigma^2$: Letting $E_{n\times n}$ denote the identity matrix in $\R^{n\times n}$ we set 
\begin{align*}
    \tilde A:= \begin{pmatrix}
    A\\ E_{n\times n}
    \end{pmatrix},\quad \tilde y:= \begin{pmatrix}
    y\\ 0_{n}
    \end{pmatrix},\quad \tilde \eps := \begin{pmatrix}
    \eps\\ \eta
    \end{pmatrix}
\end{align*}
where $\eta\sim \mathcal N(0, \Sigma)$ (the authors use the notation $C_0$ for the covariance of the artificial noise on parameter space, which clashes with our notation) such that $\tilde \eps \sim \mathcal N(0, \tilde \Gamma)$ with
\begin{align*}
    \tilde \Gamma = \begin{pmatrix}
    \Gamma & 0_{m\times n}\\ 0_{n\times m} & \lambda^{-1} \Sigma
    \end{pmatrix}
\end{align*}
Consequently, the dynamics of {sample} mean and covariance in the deterministic case (which is referred to as TEKI-flow by the authors of \cite{chada2020tikhonov}) are precisely in the form of \labelcref{eq:ODE_cov_emp,eq:ODE_mean_emp} and the corresponding mean-field and average stochastic equations are in the form of \labelcref{eq:ODE_cov_mean_field,eq:ODE_mean_mean_field,eq:ODE_cov_av_emp,eq:ODE_mean_av_emp}, respectively.
Hence, all the results of this manuscript can be directly applied to TEKI by replacing $A$ with $\tilde A$ etc.
In particular, our findings on the long time behavior from \cref{sec:particledynamics} are consistent with what was proved for TEKI in \cite{chada2020tikhonov}. 
{In a similar vein, we suspect our results to carry over (correcting for different numerical factors and additional bias terms) to a more general family of Ensemble Kalman-type methods, as the ones collected in \citep{huang2022efficient,huang2022iterated,calvello2022ensemble}.}
\end{remark}}
\paragraph{{Closed-form solutions of mean and covariance}}
As shown in \cite{garbuno2020interacting}, the solution of the mean-field equations \labelcref{eq:ODE_cov_mean_field,eq:ODE_mean_mean_field} can be explicitly computed and is given by
\begin{align}
    \label{eq:sol_cov_mean_field}
    \mathfrak C(t) &= \mathfrak C_0(E+tA^T\Gamma^{-1}A\mathfrak C_0)^{-1},\\
    \label{eq:sol_mean_mean_field}
    \mathfrak m(t) &= \mathfrak C(t)\mathfrak C_0^{-1}\mathfrak m_0 + t \mathfrak C(t)A^T\Gamma^{-1}y.
\end{align}
{where we define $E$ as the identity matrix in $\R^{m\times m}$.}
For $t=1$ these {quantities} coincide with the posterior mean and covariance associated to \labelcref{eq:inverse_prob}, which can be explicitly computed in this Gaussian and linear regime.
In contrast, as we will show in this paper that the solutions to the other systems of equations \emph{do not} recover the posterior exactly.
For instance, the solution to \labelcref{eq:ODE_cov_emp,eq:ODE_mean_emp} is given by
\begin{align}
    C(t) &=  C_0(E+2tA^T\Gamma^{-1}A C_0)^{-1},\\
    \label{eq:sol_mean_emp}
    m(t) &= \sqrt{ C(t) C_0^{-1}}  m_0 +  \sqrt{ C(t) C_0^{-1}} \int_0^t \sqrt{ C(s) C_0^{-1}}\d s\, C_0A^T\Gamma^{-1}y,
\end{align}
which does clearly not coincide with the posterior mean and covariance for any time $t\geq 0$.

{
\begin{remark}
{It is important to keep in mind} the fundamental philosophical distinction between equations \labelcref{eq:ODE_cov_emp,eq:ODE_mean_emp} on the one hand and equations \labelcref{eq:ODE_cov_mean_field,eq:ODE_mean_mean_field,eq:ODE_cov_av_emp,eq:ODE_mean_av_emp} on the other hand: The latter equations stem from the viewpoint of EnSMC, with the goal of transporting the prior to the posterior in one artificial time unit, which is why they manage to do exactly that. Meanwhile, the former equations come from the philosophy of optimization/regularized inversion, which is why they lack the property of recovering the posterior measure. The recently developed Ensemble Kalman sampler (EKS) \cite{garbuno2020interacting} is in some sense very similar to the EnSMC method, but is designed in order to approximate the posterior measure (in a Langevin-dynamics sense) for $t\to\infty$ rather than fitting this posterior exactly for $t=1$. In the linear setting, the mean and covariance of its mean-field limit are governed by
\begin{align}\label{eq:EKS_cov}
\dot{\mathfrak{C}}_{\mathrm{EKS}}(t) &=
-2\mathfrak{C}_{\mathrm{EKS}}(t)\left(A^T\Gamma^{-1}A + C_0^{-1}-\mathfrak{C}_{\mathrm{EKS}}^{-1}(t)\right) \mathfrak{C}_{\mathrm{EKS}}(t),\\
\label{eq:EKS_mean}
\dot {\mathfrak{m}}_{\mathrm{EKS}}(t) &= -\mathfrak{C}_{\mathrm{EKS}}(t)\left((A^T\Gamma^{-1}A + C_0^{-1})\mathfrak{m}_{\mathrm{EKS}}(t)-A^T\Gamma^{-1}y\right).
\end{align}
Here, $C_0$ is the prior measure's covariance (corresponding to $\Sigma$ in the TEKI methodology and $C_0$ for EKI).
It is obvious from \labelcref{eq:EKS_cov,eq:EKS_mean} that at convergence EKS yields the posterior covariance and mean.

{The relationship between EKS and EKI is very similar to how Langevin dynamics relates to gradient descent: The incorporation of prior information and noise acts as a counterpart to the contraction of covariance, which is why EKS approximates the posterior while EKI contracts around a candidate for the optimization problem.}
\end{remark}
}

\paragraph{A unified framework for ensemble analysis}

In this article we consider the following two differential equations
\begin{alignat}{2}
    \label{eq:ODE_cov}
    \dot{\cov}(t) &= -\alpha\cov(t) A^T \Gamma^{-1} A \cov(t), \quad &&\cov(0) = \cov_0,\\
    \label{eq:ODE_mean}
    \dot{\mean}(t) &= -\cov(t) A^T\Gamma^{-1}(A\mean(t)-y), \quad &&\mean(0) = \mean_0,
\end{alignat}
where $\alpha\geq 1$ is a free parameter. 
For $\alpha=2$ they reduce to the deterministic equations \labelcref{eq:ODE_cov_emp,eq:ODE_mean_emp} of the sample mean and covariance, for $\alpha=1$ to the mean-field equations \labelcref{eq:ODE_cov_mean_field,eq:ODE_mean_mean_field}, and for $\alpha=\frac{J+1}{J}$ to {the leading term of} the average sample equations \labelcref{eq:ODE_cov_av_emp,eq:ODE_mean_av_emp}.

With this we continue a long line of articles which analyze fine deterministic properties of EKI.
For instance, we refer to \cite{schillings2017analysis,schillings2018convergence,chada2020tikhonov,parzer2021convergence} for results, e.g., on existence, asymptotic behavior, and zero-noise consistency of EKI.
In the whole article we assume that $\cov_0$ is  symmetric and positive definite.
We record the most important {components} of our notation in \cref{tab:notation}.

As can be easily checked the ordinary differential equation \labelcref{eq:ODE_cov} for the covariance has the explicit solution
\begin{align}
    \label{eq:sol_cov}
    \cov(t) = \cov_0(E+\alpha tA^T\Gamma^{-1}A\cov_0)^{-1},
\end{align}
which is a simple time-rescaling of the mean-field solution \labelcref{eq:sol_cov_mean_field}.
However, as already seen above, the equation for the mean $\mean(t)$ cannot be obtained as simple as that.

In contrast, the formula for the sample mean which we obtain is given by the expression
\begin{align}\label{eq:sol_mean}
    \mean(t) = \left(\cov(t)\cov_0^{-1}\right)^\frac{1}{\alpha}\mean_0 + \left(\cov(t)\cov_0^{-1}\right)^\frac{1}{\alpha}\int_0^t\left(\cov(s)\cov_0^{-1}\right)^{1-\frac{1}{\alpha}}\d s \,\cov_0A^T\Gamma^{-1}y
\end{align}
and will be derived in \cref{sec:char_ode}.
Obviously, for $\alpha=1$ this coincides with \labelcref{eq:sol_mean_mean_field} and for $\alpha=2$ with \labelcref{eq:sol_mean_emp}.
We want to point out, though, that independent of $\alpha$ the asymptotic behaviour of particles and their covariance as $t\to\infty$ is the same since we can show that $\lim_{t\to\infty}\cov(t)\cov_0^{-1} =\lim_{t\to\infty}\left(\cov(t)\cov_0^{-1}\right)^p$ for every power $p>0$.

\begin{table}
\begin{tabular}[h]{l|l}
Notation & Meaning  \\
\hline
\hline
$C, m$ & \begin{minipage}{0.8\textwidth}{sample} covariance and mean of a given ensemble for EKI ({naive} or stochastic).\\[-0.8em]\end{minipage}\\
\hline
$\mathbf C, \mathbf m$ & \begin{minipage}{0.8\textwidth}Average (w.r.t. the Wiener measure $\{\wiener^j\}_{j}$) {sample} covariance and mean of a given ensemble for EKI.\\[-0.8em]
\end{minipage}\\
\hline
$\mathfrak C, \mathfrak m$ & Mean and covariance of mean-field limit of stochastic EKI.\\
\hline
$\cov, \mean$ & \begin{minipage}{0.8\textwidth}
Stand-in for $(C,m),(\mathfrak C, \mathfrak m)$ and $(\mathbf C, \mathbf m)$, respectively, for a unified calculation valid for all special cases.\\[-0.8em]\end{minipage} \\
\end{tabular}
    \caption{Notation in this manuscript}
    \label{tab:notation}
\end{table}

\paragraph{Time-invariant diagonalization of $\cov(t)\cov_0^{-1}$}

Although we will use the explicit expression \labelcref{eq:sol_cov} for $\cov(t)$ in the following, we want to point out that many statements given can be rephrased in a form which only requires knowledge of the following fact: 
While $\cov(t)$ changes its diagonalization over time (see \cref{sec:spectral}), the eigenvectors of $\cov(t)\cov_0^{-1}$ stay constant. 
This can be seen by defining $\cov^c(t) = \cov(t) \cov_0^{-1}$, by which \labelcref{eq:ODE_cov} becomes $\dot\cov^c(t) = -\alpha\cov^c(t) \cov_0A^T\Gamma^{-1}A \cov^c(t)$. 
If we diagonalize $\cov_0 A^T\Gamma^{-1}A = SDS^{-1}$ (which we justify in the following) and set $\cov^d(t) := S^{-1}\cov^c(t) S$, we obtain the matrix ordinary differential equation
\[ \dot \cov^d(t) = -\alpha \cov^d(t) D \cov^d(t).\]
This decouples into a set of $n$ scalar ordinary differential equations since $\cov^d(0) = E$ and $D$ is a diagonal matrix.

\paragraph{Approximating the posterior for $t=1$}

Related to the ensemble Kalman inversion is the Bayesian approach of solving the inverse problem $Au=y$ by computing the Maximum A Posteriori (MAP) estimator
\begin{align}\label{eq:MAP}
    u_\mathrm{MAP}(t) = \argmin \frac{t}{2} \norm{Au-y}_\Gamma^2 + \frac{1}{2} \norm{u-m_0}_{C_0}^2,
\end{align}
where the prior $u \sim \mathcal{N}(m_0,C_0)$ plays the role of a Tikhonov-type regularization.
Using the optimality conditions of \labelcref{eq:MAP}, one can explicitly compute the MAP estimator as
\begin{align}\label{eq:MAP_sol}
     u_\mathrm{MAP}(t) = m_0 + t C_0 (E +t A^T \Gamma^{-1} A C_0)^{-1}A^T \Gamma^{-1}(y-A m_0).
\end{align}
Here $t>0$ is a free parameter which is usually set to $1$.
The MAP estimator has some heuristic relations to the discrete ensemble Kalman inversion \labelcref{eq:enk_it}.
In the light of \labelcref{eq:enk_it_means_optim} it is obvious that the MAP estimator can be seen as one-step explicit Euler discretization of ordinary differential equation \labelcref{eq:ODE_mean} which drives the sample means $t\mapsto m(t)$ of the time-continuous ensemble Kalman inversion \labelcref{eq:enkf_stoch}.
In particular, if one performs one iteration of \labelcref{eq:enk_it} with $\tau=1$ and $\Sigma=0$, then the sample mean coincides with $u_\mathrm{MAP}$ for $t=1$. {This can be seen by comparing \labelcref{eq:MAP_sol} for $t=1$ and the sample mean of \labelcref{eq:enk_it} for $\tau=1$ and $\Sigma=0$ (i.e., unperturbed observation) side by side. This is under the condition that the initial ensemble $\{u_0^j\}_{j=1}^J$ has the same statistical properties (sample mean and covariance) as the Bayesian prior $\mu_0$. This is interesting in light of the fact that this property is violated for the multi-step case (i.e., $\tau\neq 1$), and also its continuous limit given by \labelcref{eq:enkf_stoch}, both in the stochastic case as well as for naive EKI ($\Sigma = 0$).}

One might hope that a similar property is true for the ensemble mean $\mean(t)$, governed by \labelcref{eq:ODE_cov,eq:ODE_mean}.
This is unfortunately not the case for general $\alpha>0$, as can be seen from the explicit solution \labelcref{eq:sol_cov,eq:sol_mean}, which we derive.
It can be seen that, with the only exception of $\alpha=1$, there is no time $t\geq 0$ such that $\cov(t)$ and $\mean(t)$ constitute the mean and covariance of the posterior, not even in the linear Gaussian setting.
The fundamental problem here is that one-step time discretizations are generally not exact.
For readers interested in this issue we refer to \cite{burger2016spectral,bungert2019nonlinear}, where it was classified when one-step time discretizations and gradient flows of one-homogeneous functionals coincide.

\begin{figure}[t]
    \centering
    \includegraphics[width=0.49\textwidth,trim=1cm .5cm 1.5cm 1cm,clip]{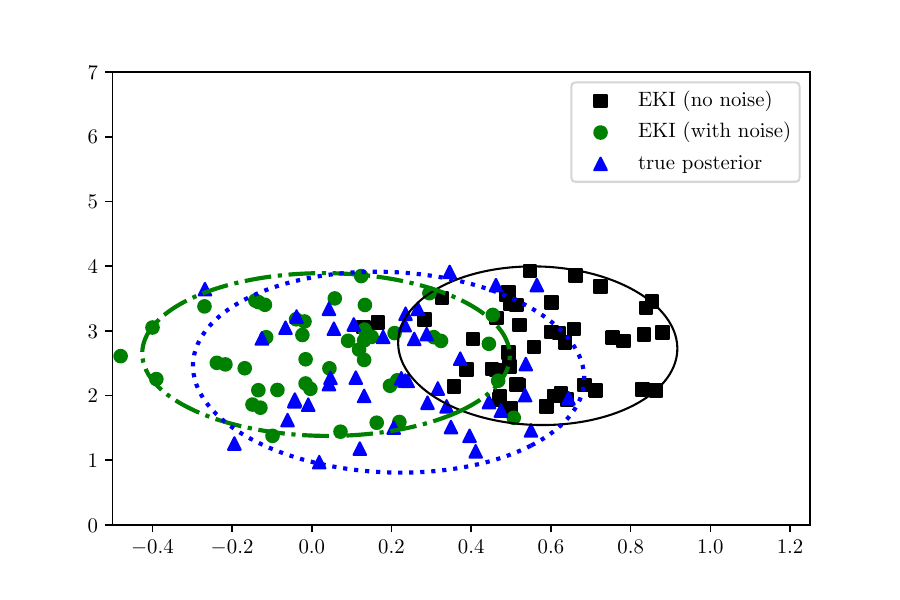}
    \hfill%
    \includegraphics[width=0.49\textwidth,trim=1cm .5cm 1.5cm 1cm,clip]{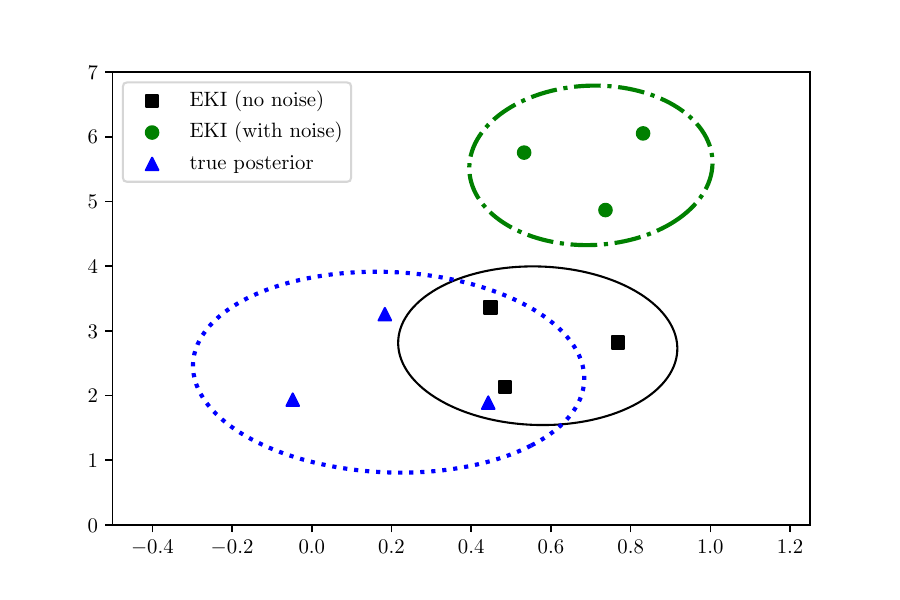}\\
    \centering
    \hfill%
    \caption{Implementation of EKI dynamics for ensemble size 45 (left) and 3 (right) for time $t=1$. \textbf{Blue:} Ensemble matching the true posterior (triangles, dotted line). \textbf{Black:} Ensemble from {naive} EKI for $t=1$ (squares,  solid line). \textbf{Green:} One realization of an ensemble from the stochastic EKI (circles, dash-dot line). This EKI approximately recovers the posterior measure, but only for large ensemble sizes (due to the noise and the fact that $J=3$ is not sufficiently close to the mean-field limit). Approximating the mean-field limit, i.e., taking a large ensemble $J=45$ improves the quality of approximation. Note that mean and covariance of {naive} EKI are identical between the two figures as they do not depend on $J$, unlike the stochastic EKI. Axes are scaled for visualization purposes. Ellipses visualize one standard deviation of the {sample} covariances involved. }
    \label{fig:covariances}
\end{figure}

For illustration purposes, we compare an ensemble of particles which matches the true posterior with the result of the EKI for $t=1$ with and without noise perturbation in \cref{fig:covariances}.
The set-up here is $A = \diag(4,1)$, the Gaussian prior has mean and covariance
\begin{align*}
    \mean_0 = 
    \begin{pmatrix}
    4\\4
    \end{pmatrix},\quad
    \cov_0 = 
    \begin{pmatrix}
    2 & -1 \\ -1 & 2
    \end{pmatrix},
\end{align*}
and data is given by $y=(0,0)^T$.
One can clearly see that the {sample} mean and covariances of the unperturbed EKI are quite different from the posterior samples. The perturbed EKI approximates the posterior better but also not exactly. This is due to two reasons: First, there is perturbative noise in the stochastic EKI dynamics, so the ensemble cannot match the posterior exactly due to stochasticity (only the mean-field dynamics exactly recovers the posterior for $t=1$). Second, the factor of $\alpha = (J+1)/J$ is not equal to $1$ for finite ensemble sizes, hence even the average (w.r.t. Wiener measure) dynamics is not identical to the posterior-recovering dynamics \labelcref{eq:ODE_cov_mean_field,eq:ODE_mean_mean_field}. 
However, with increasing number of particles the approximation gets better since one then enters the mean-field regime, which recovers the posterior.

\paragraph{Ensemble collapse and non-monotonous residual convergence}
Besides sample covariance and mean it is also interesting to consider the \emph{deviations} $e^j(t)$, the \emph{residuals} $r^j(t)$, and the \emph{residual mean} $r(t)$ of the {naive} EKI dynamics given by \labelcref{eq:enkf_stoch} for $\Sigma = 0$, i.e., 
\begin{equation*} 
     \dot{u}^j(t) = - C(t) A^T\Gamma^{-1}(Au^j(t)-y).
\end{equation*}
Deviations, residuals, and residual mean are defined through
\begin{align*}
    e^j(t) &:= u^j(t)-m(t),\\
    r^j(t) &:=u^j(t) - {u},\\
    r(t) &:= \frac{1}{J}\sum_{j=1}^J r^j(t) = m(t) - {u},
\end{align*}
respectively.
Here, ${u}$ is a suitable parameter value with respect to the data which will be specified in \cref{sec:particledynamics}.
If $A$ has a nontrivial kernel, there is a whole subspace of possible candidates for ${u}$. We will see that there is a canonical choice depending on the initial ensemble $\{u^j_0\}_{j=1}^J$.

Naturally, one can also study similar quantities for stochastic or mean-field EKI (with very similar techniques) but this is beyond the scope of this paper.

Following from \labelcref{eq:ODE_cov_emp,eq:ODE_mean_emp} deviations, residuals, and residual mean evolve according to the same differential equation 
\begin{align}\label{eq:ODE_residuals}
    \dot{x}(t) &= -C(t) A^T\Gamma^{-1}A x(t),\quad x\in\{e^j,r^j,r\},\\
    \dot C(t) &= -2C(t)A^T\Gamma^{-1}AC(t),
\end{align}
and differ only in their initialization $x_0\in \{e^j_0, r^j_0,r_0\}$. 

As we will see, the explicit solution to \labelcref{eq:ODE_residuals} with initial condition $x(0)=x_0$ is given by
\begin{align}\label{eq:sol_ODE_residuals}
    x(t) = \sqrt{C(t)C_0^{-1}} x_0.
\end{align}
Hence, for understanding the asymptotic behavior of the deviations, residuals, and residual mean, it suffices to characterize the asymptotic behavior of the {sample} covariance matrix $C(t)$, or rather the preconditioned covariance matrix $C(t)C_0^{-1}$.

Of particular interest in the context of ensemble Kalman inversion are the ensemble and residual spreads, defined as
\begin{align}
    \label{eq:ensemble_spread}
    V_e(t) = \frac{1}{2J}\sum_{j=1}^J \norm{e^j(t)}^2,\\
    \label{eq:residual_spread}
    V_r(t) = \frac{1}{2J}\sum_{j=1}^J \norm{r^j(t)}^2.
\end{align}
It is well-known that the ensemble spread $V_e$ is monotonically decreasing (this is the deterministic analogue of Theorem~3.4 in \cite{blomker2019well}), which we call \textit{ensemble collapse}.
In contrast, the residual spread $V_r$ does not decrease monotonously in general (even in the zero-noise setting and for the optimal choice of ${u}$ in the definition of $r^j$) and we devote \cref{sec:spreads} to a detailed study of this issue.

Mapping the deviations and residuals into observation space, one can define the functions
\begin{align}
    \label{eq:fwd_ensemble_spread}
    \mathfrak V_e(t) = \frac{1}{2J}\sum_{j=1}^J \norm{Ae^j(t)}_\Gamma^2,\\
    \label{eq:fwd_residual_spread}
    \mathfrak V_r(t) = \frac{1}{2J}\sum_{j=1}^J \norm{Ar^j(t)}_\Gamma^2,
\end{align}
which indeed decrease monotonously, as shown in \cite{schillings2017analysis}.

As far as the authors are aware, there has not been any exhaustive discussion of the convergence of deterministic particle dynamics and residuals of the ensemble Kalman inversion method for the infinite-time limit in parameter space: 
The introduction of the ensemble Kalman methodology was carried out in \cite{iglesias2013ensemble}, with no analysis of its dynamics for $t\to \infty$.
In \cite{kelly2014well} it was proved (although in a more general data assimilation setting) that the dynamics do not blow up in finite time by bounding their growth exponentially, but numerical experiments soon suggested that a lot more would be feasible. 
The first long-time analysis of the EKI for Bayesian inverse problems (outside of the physically time-dependent data assimilation domain) was conducted in \cite{schillings2017analysis}, with a sequel in \cite{schillings2018convergence}, but is constrained to the observation space which -- for the case of lower-rank maps $A$ -- only allows control on a subspace of the parameter space (given by the orthogonal complement of $A$).
Similarly, \cite{blomker2018strongly} and \cite{blomker2019well} proved convergence to zero of the ensemble spread but could prove convergence of the residuals only in observation space and under additional variance inflation. 
The missing piece, i.e., a full deterministic convergence analysis of the ensemble Kalman inversion methodology for linear Bayesian inverse problems, is provided with this manuscript.

\begin{remark}[subspace property]\label{rem:infdim}
{In \cite{emerick2013ensemble,iglesias2013ensemble}}, it was shown that the EKI dynamics (both in {naive} and stochastic version) stays in the affine subspace spanned by the initial ensemble. This means that even for a infinite-dimensional inverse problem, the EKI method lives in finite-dimensional space which is why we assume this a-priori. The setting of ensemble Kalman inversion on an infinite-dimensional parameter space can be rephrased as a finite-dimensional problem by describing all quantities involved by the span of the initial ensemble. By the subspace property, we never leave this finite-dimensional span. 
Most realistic measurement processes generate finite-dimensional even for infinite-dimensional parameter models. This means that assuming both parameter and observation to be finite-dimensional quantities is not a strong restriction in most contexts.

{%
Furthermore, the subspace property allows one to restrict all considerations to the affine subspace, spanned by the initial ensemble, where the covariance matrix $C_0$ is invertible.
Hence, we assume the latter for the rest of the paper.
Furthermore, all ground truth quantities must be interpreted as being projected onto this affine subspace of the parameter space.}
\end{remark}

We also remark that the ``time parameter'' $t$ rather corresponds to a regularization parameter and plays the role of supposedly (but in a strict sense only for the mean-field limit of the stochastic EKI) interpolating between prior ($t=0$) and posterior ($t=1$), with $t\to\infty$ corresponding to the limit of vanishing regularization. 
This means that $t$ is not physical time and does also not keep track of data accumulating over time, as it is in the setting of data assimilation. 
Relevant research on the ensemble Kalman--Bucy filter and the ensemble Kalman inversion and their respective long-time behaviour include \cite{kelly2014well,amezcua2014ensemble,tong2016nonlinear,de2018long,ding2019mean,chada2020tikhonov}, with \cite{bergemann2009ensemble,bergemann2010localization,bergemann2010mollified,bergemann2012ensemble} being particularly relevant to our approach of analysing the ensemble's {sample} mean and covariance.
For a very well-written and extensive review of research on the stochastic ensemble Kalman methodology and its interpretation as metric gradient flow, we refer to \cite{garbuno2020interacting} and the references therein.


\section{Spectral decomposition of preconditioned covariance}\label{sec:spec_precond}

We recall that we want to analyze \labelcref{eq:ODE_cov}, which is
\begin{align*}
    \dot{\cov}(t) = -\alpha\cov(t) A^T \Gamma^{-1} A \cov(t), \quad \cov(0) = \cov_0
\end{align*}
with $\alpha > 0$ being a parameter unifying the treatment of {naive}, average stochastic, and mean-field EKI.
We will use the closed-form solution~\labelcref{eq:sol_cov}
\begin{align*}
    \cov(t) = \cov_0(E+\alpha tA^T\Gamma^{-1}A\cov_0)^{-1}
\end{align*}
in order to derive a spectral decomposition of $\cov(t)\cov_0^{-1}$ and study asymptotic profiles.

Since the explicit solution alone does not provide us with sufficient insight about the behavior of $\cov(t)$, we pursue a different strategy for constructing a solution. 
Additionally, we hope that the methods used will extend to cases where explicit solutions can not be constructed (e.g., the nonlinear or stochastic case).

The following two lemmas will be needed for the construction. {They are not novel results, but for completeness of this manuscript we record their proofs here}.
We note that in this manuscript we will consider positive definite matrices which are not necessarily symmetric.
{The following result is well-known, albeit a proof is hard to find. For completeness we give a proof below.}

\begin{lemma}\label{lem:product_diagonalizable}
Given two symmetric matrices $V,W\in\R^{n\times n}$ with at least one of them being positive definite, the products $VW$ and $WV$ are diagonalizable. If the other matrix is also at least positive (semi)definite, then $VW$ and $WV$ are also positive (semi)definite.
\end{lemma}
\begin{proof}
Without loss of generality assume that $W$ is positive definite. Thus there exists a non-singular square root matrix $W^{1/2}$. Then
\[W^{-1/2}WVW^{1/2} = W^{1/2}VW^{1/2}.\]
The matrix $W^{1/2}VW^{1/2}$ is symmetric and thus diagonalizable. The left hand side is a similarity transform of the product $WV$. This shows that $WV$ is diagonalizable. The statement is proven for $VW$ by arguing the same way for $W^{1/2}VWW^{-1/2}$.
For the other statement note that if $V$ is positive (semi)definite in addition, then there also exists a square root $V^{1/2}$ and then $W^{1/2}VW^{1/2} = (W^{1/2}V^{1/2})(V^{1/2}W^{1/2})$ which is positive (semi)definite and hence this property holds for the matrices $VW$ and $WV$. 
\end{proof}

\begin{lemma}\label{lem:algebraic_sol}
Let $M$ be diagonalizable with non-negative eigenvalues such that 
$$M = S\diag(\mu_1,\ldots,\mu_k, 0, \ldots, 0)S^{-1}.$$
The columns of $S=(w_1,\ldots,w_n)$ are the eigenvectors such that $Mw_i = \mu_i w_i$.

Then $R(t) = (E+tM)^{-1}$ exists for all $t\in \R\setminus\{-\mu_1^{-1},\dots,-\mu_n^{-1}\}$ and has the form 
\begin{align*}
    R(t) = S\diag\left(\frac{1}{1+t\mu_i}\right)_{i=1}^n S^{-1}.
\end{align*}
Also, $R_\infty := \lim_{t\to\infty}R(t)$ exists and has the same eigenvectors as $M$, but with 
\begin{align*}
    R_\infty w_i = 
    \begin{cases}
        0&\text{ for } i =1,\ldots, k,\\
        w_i&\text{ for } i=k+1,\ldots,n.
    \end{cases}
\end{align*}
\end{lemma}
\begin{proof}
The invertibility of $(E+tM)$ for $t\neq -\mu_i^{-1}$, $i=1,\ldots,k$, follows directly from the spectral properties of $M$. Note that 
\[R(t) = (E+t S D S^{-1})^{-1} = S(E+tD)^{-1}S^{-1} = S\diag\left(\frac{1}{1+t\mu_i}\right)_{i=1}^n S^{-1}.\]
This proves that indeed $R(t)$ has the same spectral decomposition as $M$ for all $t\in\R$ with
\[R(t)w_i = \frac{1}{1+t\mu_i}w_i\]
and the limit $t\to\infty$ as claimed.
\end{proof}

\subsection{Diagonalization of the preconditioned covariance}\label{sec:diagonalization}
Now we are ready to formulate the first theorem of this section. 
The main idea here is that, while an eigenbasis of $\cov(t)$ seems to behave erratically (stretching and rotating in a way to assimilate both the current covariance structure and the influence of the observation operator), the matrix $\cov(t)\cov_0^{-1}$ has fixed eigenvectors which equal those of $\cov_0 A^T\Gamma^{-1}A$ (albeit with different eigenvalues). 
This already follows from \cite{garbuno2020interacting}, where it is shown that $\cov(t)^{-1}$ is linear in time and that $\cov(t)^{-1}\cov_0$ has eigenvectors which do not depend on $t$ and allows to derive an explicit form of the solution of the EKI dynamics.
\begin{theorem}[covariance dynamics]\label{thm:cov_dynamics}
Let $\cov(t) = (E+\alpha t\cov_0A^T\Gamma^{-1}A)^{-1}\cov_0$ denote the solution of \labelcref{eq:ODE_cov} with initial condition $\cov_0$.
Then it holds:
\begin{itemize}
    \item The matrix $\cov_0 A^T\Gamma^{-1}A$ is diagonalizable, meaning that $\cov_0 A^T\Gamma^{-1}A = SDS^{-1}$ with $D = \diag(\mu_1,\dots,\mu_k,0,\dots,0)$ where $\mu_1\geq \cdots \geq\mu_k>0$ for some $k\leq n$.
    \item Letting $E(t):=\diag\left(\frac{1}{1+\alpha t\mu_i}\right)_{i=1}^n$, it holds that
    \begin{align*}
        \cov(t) &= S E(t) S^{-1} \cov_0,\\
        \cov(t)A^T\Gamma^{-1}A &= SD(t)S^{-1},\\
        \cov_0A^T\Gamma^{-1}A\cov(t) &= SD(t)S^{-1}\cov_0,
    \end{align*}
    where $D(t) := DE(t) = \diag\left( \frac{\mu_i}{1+\alpha t\mu_i}\right)_{i=1}^n$. Note that $D(0) = D$ and $E(0) = E$. 
    \item $\cov_\infty :=\lim_{t\to\infty}\cov(t) = SE_\infty S^{-1} \cov_0$, where $E_\infty = \diag(0,\ldots,0,1,\ldots, 1)$ has $k$ zero entries and $n-k$ entries of one, has the property that \[A\cov_\infty=0.\] 
\end{itemize}
\end{theorem}
\begin{proof}
Using \cref{lem:product_diagonalizable} the matrix $\cov_0A^T\Gamma^{-1}A$ is positive semidefinite and diagonalizable as $\cov_0A^T\Gamma^{-1}A = S \diag(\mu_1,\ldots,\mu_k, 0, \ldots, 0)S^{-1}$.
From \labelcref{eq:sol_cov} we have an explicit solution of the covariance ordinary differential equation~\labelcref{eq:ODE_cov} given by $\cov(t) = (E+\alpha t\cov_0A^T\Gamma^{-1}A)^{-1}\cov_0$. Hence, by \cref{lem:algebraic_sol} we can express this as $\cov(t)= S E(t) S^{-1}\cov_0$ for all $t\geq 0$, with 
\begin{align*}
E(t)=\diag\left(\frac{1}{1+\alpha t\mu_i}\right)_{i=1}^n    
\end{align*}
as claimed. The characterization of $\cov(t)A^T\Gamma^{-1}A$ follows directly by seeing that
\begin{align*}
    \cov(t)A^T\Gamma^{-1}A &= SE(t)S^{-1}\cov_0A^T\Gamma^{-1}A = SE(t)DS^{-1} = SD(t)S^{-1}.
\end{align*}
Using $D E_\infty = 0$, one calculates
\begin{align*}
    \cov_\infty A^T\Gamma^{-1}A &= SE_\infty S^{-1} \cov_0 A^T\Gamma^{-1}A = SE_\infty S^{-1}SDS^{-1} = SE_\infty DS^{-1} = 0.
\end{align*}
This means that, by transposing this equality, we also have $A^T\Gamma^{-1}A\cov_\infty = 0$, which is equivalent to $A \cov_\infty=0$.
\end{proof}

\subsection{Asymptotic profiles}

In the previous section we have used a diagonalization to construct an explicit solution of the covariance matrix $\cov(t)$ and understand its asymptotic behavior as $t\to\infty$. 
Now we study ``second-order'' asymptotics by investigating the asymptotic behavior of the time derivative $\dot{\cov}(t)$ of the covariance.
More precisely, we study asymptotic profiles of $\cov(t)$ in the spirit of \cite{bungert2019asymptotic}, which are defined as limit of the approximate time derivative at infinity
\begin{align*}
    \lim_{t\to\infty}t(\cov(t) - \cov_\infty),
\end{align*}
typically solve a nonlinear eigenvalue problem, and constitute self-similar solutions of the dynamics.

To set the scene, we rewrite the covariance dynamics \labelcref{eq:ODE_cov} in the abstract form
\begin{align*}
    \dot{\cov}(t) = -\mathcal{A}(\cov(t)), \qquad \cov(0)=\cov_0,
\end{align*}
where the nonlinear operator $\mathcal{A}$ is defined as
\begin{align}\label{eq:operator_A}
    \mathcal{A}:\R^{n\times n}\to\R^{n\times n}, \quad
    \mathcal{A}(\cov) =\alpha \cov A^T \Gamma^{-1} A \cov.
\end{align}
We will show that the rescaled solutions $t(\cov(t)-\cov_\infty)$ converge to an eigenvector $\hat{\cov}$ of $\mathcal{A}$, meaning $\hat{\cov}=\mathcal{A}(\hat{\cov})$

Note that eigenvectors of $\mathcal{A}$, i.e., matrices with $\lambda\hat{\cov}=\mathcal{A}(\hat{\cov})$ give rise to self-similar solutions of \labelcref{eq:ODE_cov} (cf.~\cite{bungert2019asymptotic} for a general study).
Indeed if $\cov(0)=\hat{\cov}$ then $\cov(t)=a(t) \hat{\cov}$, where $a:[0,\infty)\to\R$ solves the initial value problem $a'(t)=-\alpha\lambda a(t)^2, \;a(0)=1$. 
Hence, in this case 
\begin{align}
    \cov(t) = \frac{1}{1+\alpha \lambda t} \hat{\cov}.
\end{align}
In the context of the ensemble Kalman inversion this means that if the covariance matrix of the initial ensemble is an eigenvector of $\mathcal{A}$, then the shape of the ensemble remains unchanged during the flow. 
Hence, \cref{thm:asymptotic_profile} means that the rescaled covariance matrix of the ensemble Kalman inversion approaches a matrix which is a self-similar solution to~\labelcref{eq:ODE_cov}.

\begin{theorem}[asymptotic profiles]\label{thm:asymptotic_profile}
Let $\cov(t)$ denote the solution of \labelcref{eq:ODE_cov} with initial condition $\cov_0$.
Then the limit
\begin{align}
    \hat{\cov} := \lim_{t\to\infty}t(\cov(t)-\cov_\infty)
\end{align}
exists, and satisfies $\mathcal{A}(\hat{\cov})=\hat{\cov}$. 
\end{theorem}
\begin{proof}
Plugging in the explicit representations $\cov(t)= S E(t) S^{-1} \cov_0$ and $\cov_\infty=S E_\infty S^{-1} \cov_0$, where $E(t)$ and $E_\infty$ are as in \cref{thm:cov_dynamics}, we obtain
\begin{align*}
    t(\cov(t)-\cov_\infty) 
    &= t S \left(
    E(t)-E_\infty
    \right)
    S^{-1} \cov_0
    \\
    &=
    S
    \diag\left(\frac{t}{1+\alpha t\mu_1},\dots,\frac{t}{1+\alpha t\mu_k},0,\dots,0\right)
    S^{-1}\cov_0 \\
    &\longrightarrow 
    S
    \underbrace{
    \diag\left(\frac{1}{\alpha\mu_1},\dots,\frac{1}{\alpha\mu_k},0,\dots,0\right)
    }_{\hat{D}}
    S^{-1}\cov_0, \quad t\to\infty.
\end{align*}
Hence, we can define the matrix $\hat{\cov}:=S \hat{D}S^{-1}\cov_0$ and observe that
\begin{align*}
    \mathcal{A}(\hat{\cov}) &= \alpha \hat{\cov}A^T\Gamma^{-1}A\hat{\cov} 
    = \alpha S\hat{D}S^{-1}\cov_0 A^T\Gamma^{-1} A S\hat{D}S^{-1} \cov_0 \\
    &= \alpha S\hat{D}S^{-1} S D S^{-1} S \hat{D}S^{-1}\cov_0
    = \alpha S\hat{D} D \hat{D}S^{-1}\cov_0 
    = S\hat{D}S^{-1}\cov_0 
    = \hat{\cov}.
\end{align*}
Here we used that $\alpha \hat{D}$ is the pseudo-inverse matrix of $D$.
\end{proof}

\begin{example}
Let $A=(0,1)$, $\Gamma=E$, and $\cov_0=\begin{pmatrix} a&b \\ b&d\end{pmatrix}$. 
Then if $a\neq 0$ one can compute that 
\begin{align*}
    \cov(t) &= \frac{1}{1+\alpha t d}
    \begin{pmatrix}
    a+\alpha t\det\cov_0 & b \\
    b & d
    \end{pmatrix}, \\
    \cov_\infty &= 
    \begin{pmatrix}
    \tfrac{\det\cov_0}{d} & 0 \\
    0 & 0
    \end{pmatrix}, \\
    \hat{\cov} &= 
    \frac1\alpha
    \begin{pmatrix}
    \tfrac{b^2}{d^2} & \tfrac{b}{d} \\
    \tfrac{b}{d} & 1
    \end{pmatrix}.
\end{align*}
Without loss of generality we can assume $d=1$ which yields $\hat{\cov}=\frac1\alpha\begin{pmatrix}b^2&b\\b& 1\end{pmatrix}.$
This matrix has the eigenvectors $(1,-b)^T$ and $(b,1)^T$ with eigenvalues $0$ and $(b^2+1)/\alpha$, respectively.
This means that the ensemble approaches its limit in a tilted way if $b\neq 0$ and parallely if $b=0$.
We illustrate this in \cref{fig:profile}, where the first row shows the evolution of EKI with an initial ensemble aligned to the subspace of solutions (depicted in red), and the second row shows the evolution of a rotated ensemble.
The eigenvectors of the respective asymptotic profiles are depicted in the rightmost plots.
Here the pink eigenvectors, corresponding to the non-zero eigenvalue of $\hat{\cov}$, show the tilting with which the ensembles hits their limiting configurations.
\end{example}

\begin{figure}[tb]
    \def\Width{0.31\textwidth}
    \centering
    \fbox{\includegraphics[width=\Width,trim=8cm 6.5cm 8cm 4cm,clip]{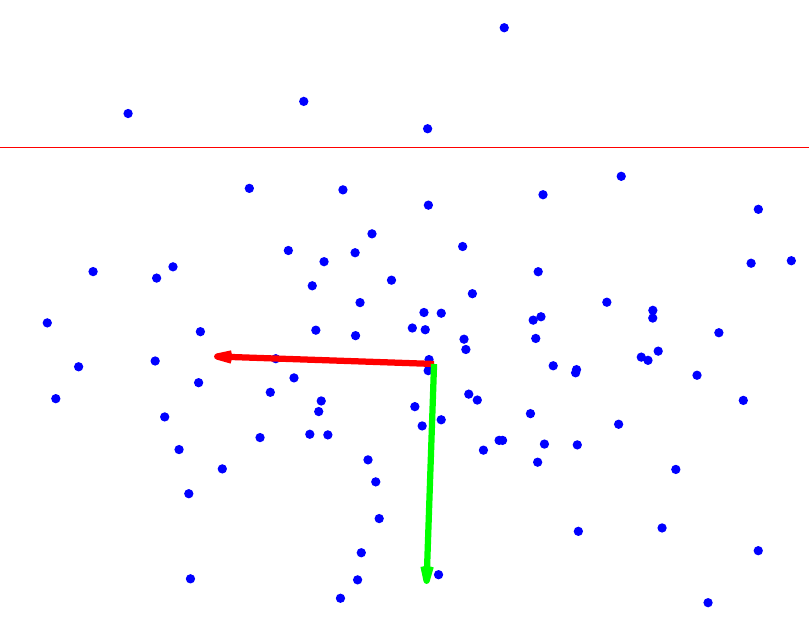}}\hfill%
    \fbox{\includegraphics[width=\Width,trim=8cm 6.5cm 8cm 4cm,clip]{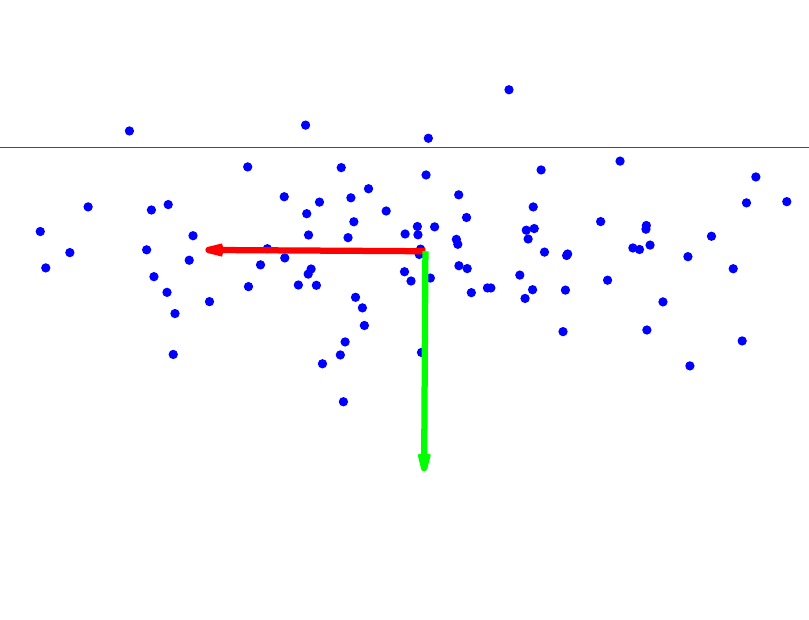}}\hfill%
    \fbox{\includegraphics[width=\Width,trim=8cm 6.5cm 8cm 4cm,clip]{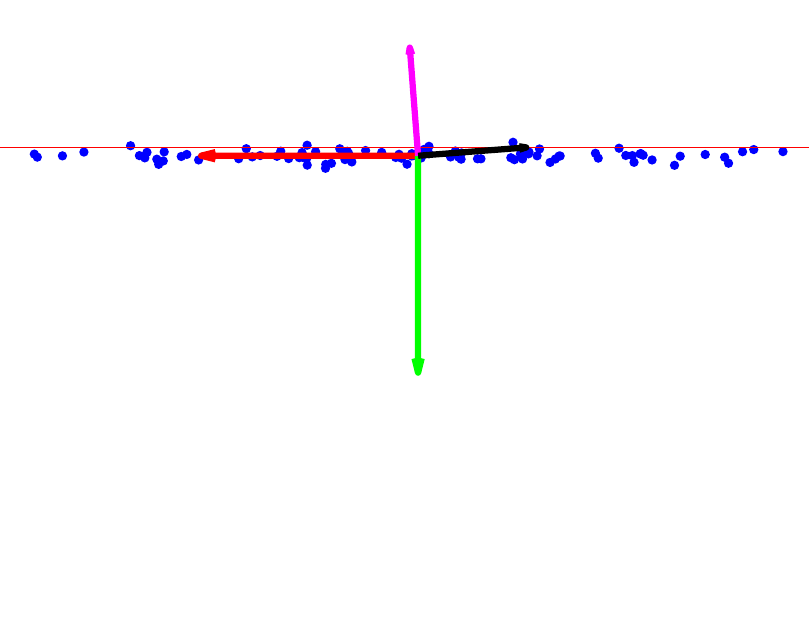}}\\%
    \vspace*{0.01\textwidth}
    \fbox{\includegraphics[width=\Width,trim=8cm 6.5cm 8cm 4cm,clip]{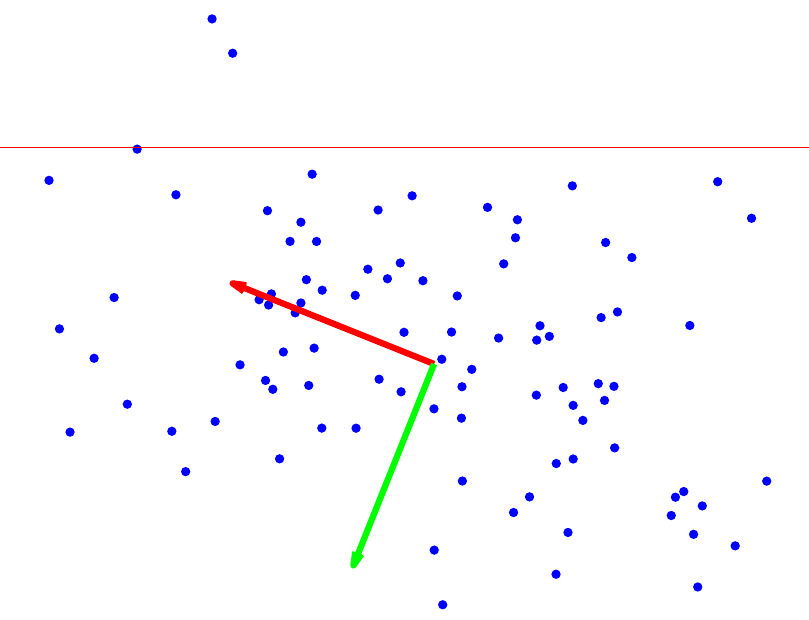}}\hfill%
    \fbox{\includegraphics[width=\Width,trim=8cm 6.5cm 8cm 4cm,clip]{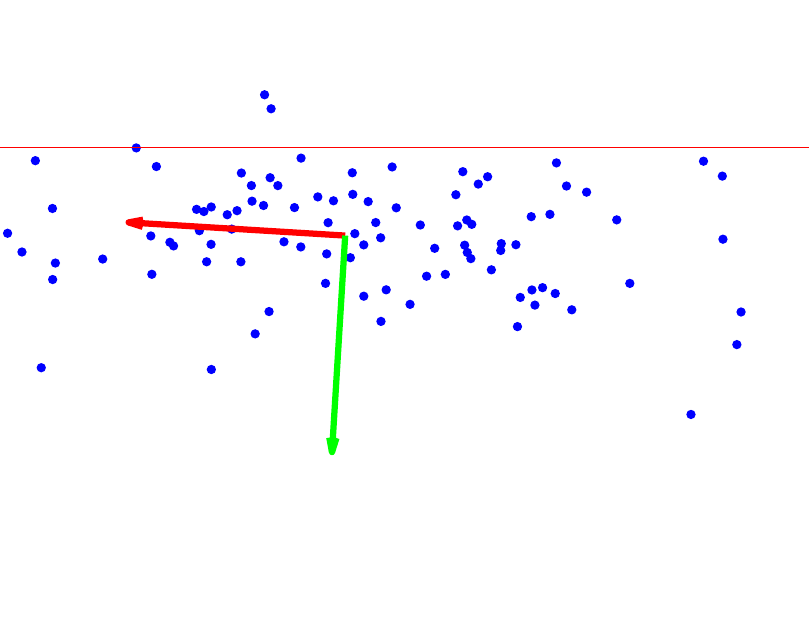}}\hfill%
    \fbox{\includegraphics[width=\Width,trim=8cm 6.5cm 8cm 4cm,clip]{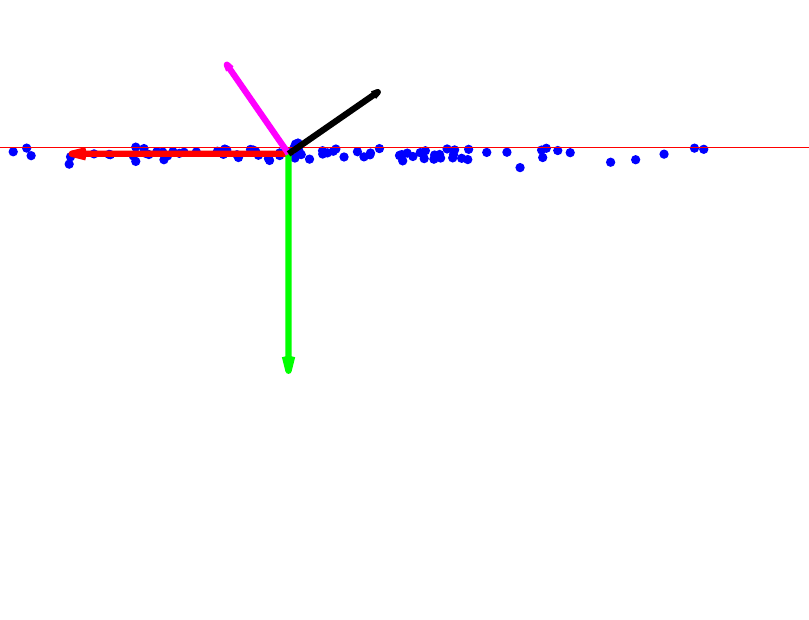}}%
    \caption{Asymptotic Profiles of ensemble Kalman inversion. From left to right: initial condition, intermediate time step, converged state. \textbf{Top row:} symmetric prior. \textbf{Bottom row:} asymmetric prior. Red and green arrows indicate eigenvectors of the {sample} covariance matrix. Magenta and black arrows indicate eigenvectors of the asymptotic profile.}
    \label{fig:profile}
\end{figure}

\section{Particle dynamics in the noisy data case}\label{sec:particledynamics}
We start again by recalling that the system we are analysing {describes} the particle dynamics of the EKI:
\begin{align*} 
     \dot{u}^j(t) = - C(t) A^T\Gamma^{-1}(Au^j(t)-y) +  C(t)A^T\Gamma^{-1}\sqrt{\Sigma}\dot{\wiener}^j(t).
\end{align*}
The evolution of particle mean and covariance (in the various senses described above) can be uniformly described by \labelcref{eq:ODE_cov,eq:ODE_mean} which we repeat for convenience:
\begin{alignat*}{2}
    \dot{\cov}(t) &= -\alpha\cov(t) A^T \Gamma^{-1} A \cov(t), \quad &&\cov(0) = \cov_0,\\
    \dot{\mean}(t) &= -\cov(t) A^T\Gamma^{-1}(A\mean(t)-y), \quad &&\mean(0) = \mean_0.
\end{alignat*}
With the complete analysis of the covariance dynamics from \cref{sec:spec_precond} at hand we can now analyse, inter alia, {the {naive} EKI dynamics of particles and their {sample} mean $m(t)$, as well as the mean-field means $\mathfrak m(t)$}. 
Furthermore, we will study the long-time behaviour of these quantities and apply our findings to ensemble and residual spreads.

We structure the main results of this section, the proofs of which follow directly from the statements in \cref{sec:char_ode,sec:asymp_rates}, into two theorems. 
The first one deals with {naive} EKI and the second one with mean-field EKI.

Before stating the theorems, we recall the definition of the pseudo-inverse of a diagonalizable matrix $M = U\Lambda U^{-1}$ as $M^+ = U\Lambda^+U^{-1}$ where $\diag(a_1\ldots,a_k,0,\ldots,0)^+ = \diag(a_1^{-1},\ldots,a_k^{-1},0,\ldots,0)$.
This allows us to define a pseudo-inverse of the diagonalizable matrix $\cov_0A^T\Gamma^{-1}A=SDS^{-1}$ as $(\cov_0A^T\Gamma^{-1}A)^+:=SD^+S^{-1}$.
For the matrix $A^T\Gamma^{-1}A$ we will therefore define $(A^T\Gamma^{-1}A)^-:=(\cov_0A^T\Gamma^{-1}A)^+\cov_0$ and remark that this is \emph{not} the Moore-Penrose pseudo-inverse since it does not fulfill a hermiticity condition. 

The Moore-Penrose pseudo-inverse $M^+$ of a matrix $M$ has to fulfill the following:
\begin{align*}
    M^+MM^+ = M^+,\quad
    MM^+M = M,\quad
    (MM^+)^T = MM^+,\quad
    (M^+M)^T = M^+M.
\end{align*}
The matrix $(A^T\Gamma^{-1}A)^-$ satisfies the first two, as can be checked by elementary computation, but not the latter two conditions, since for example
\begin{align*}
(A^T\Gamma^{-1}A)^-(A^T\Gamma^{-1}A) &= (SD^+S^{-1})\cov_0(\cov_0^{-1}SDS^{-1})
=SD^+DS^{-1}
\end{align*}
and this matrix is not symmetric in general (unless $S$ is orthogonal or $A$ is one-to-one).
\begin{theorem}[{naive} EKI]\label{thm:deterministic_EKI}
Let {$y = y^\dagger + \eps$ with $y^\dagger\in\ran(A)$} and $\{u^j\}_{j=1}^J$ solve the {naive} EKI dynamics \labelcref{eq:enkf_stoch} with $\Sigma=0$.
Then, {letting $u\in\R^n$ be an arbitrary element satisfying $Au=y^\dagger$,} the particles $u^j$ and their {sample} mean $m(t)$ satisfy:
\begin{align*}
    u^j(t) &= u^j_0 + \left(E-\sqrt{C(t)C_0^{-1}}\right)({u} - u^j_0) + \\
    &\qquad\sqrt{C(t)C_0^{-1}} \int_0^t \sqrt{C(s)C_0^{-1}}\d s\,C_0A^T\Gamma^{-1}\eps,\\
    m(t) &= m_0 + \left(E-\sqrt{C(t)C_0^{-1}}\right)({u} - m_0) + \\
    &\qquad\sqrt{C(t)C_0^{-1}} \int_0^t \sqrt{C(s)C_0^{-1}}\d s\,C_0A^T\Gamma^{-1}\eps,\\
    \lim_{t\to\infty} u^j(t) &= u^{j,\dagger} + (A^T\Gamma^{-1}A)^-(A^T\Gamma^{-1}\eps),\\
    \lim_{t\to\infty} m(t) &= m^{\dagger} + (A^T\Gamma^{-1}A)^-(A^T\Gamma^{-1}\eps),
\end{align*}
where $u^{j,\dagger} = u^j_0 + \left(E-{C_\infty C_0^{-1}}\right)({u} - u^j_0)$ and $m^{\dagger} = m_0 + \left(E-{C_\infty C_0^{-1}}\right)({u} - m_0)$  are equivalently characterized by 
\begin{align*}
u^{j,\dagger} &= \argmin\left\lbrace\|{u-u^j_0}\|_{C_0} \st u\in \R^n,\,Au = {y^\dagger}\right\rbrace, \\
m^{\dagger} &= \argmin\left\lbrace\norm{m-m_0}_{C_0}\st m\in\R^n,\,Am = {y^\dagger}\right\rbrace.
\end{align*}
Alternatively, we can write 
\begin{align*}
    \lim_{t\to\infty}u^j(t) &=  \argmin\left\lbrace\|{u-u^j_0}\|_{C_0} \st u\in \R^n,\,Au = \Pi_{\ran(A)}^\Gamma(y)\right\rbrace, \\
    \lim_{t\to\infty}m(t) &=  \argmin\left\lbrace\|{m-m_0}\|_{C_0} \st u\in \R^n,\,Am = \Pi_{\ran(A)}^\Gamma(y)\right\rbrace,
\end{align*}
where $\Pi_V^\Gamma$ is the $\Gamma$-orthogonal projection operator onto a closed subspace $V$.
The rate of convergence of all quantities involved is given by ${1}/{\sqrt{t}}$.
\end{theorem}
{
\begin{corollary}[consistency for vanishing noise]
\cref{thm:deterministic_EKI} in particular shows that as $\varepsilon\to 0$, i.e., the noise in the data vanishes, the particle positions and the sample mean converge to the minimal-prior solutions of the inverse problem.
For convergence rates of {naive} EKI in the vanishing noise limit we refer to the recent paper \cite{parzer2021convergence}.
There also a discrepancy principle for early-stopping the EKI is derived, see also \cite{iglesias2013ensemble}.
\end{corollary}}

\begin{theorem}[mean-field EKI]\label{thm:mean_field}
Let {$y = y^\dagger + \eps$ with $y^\dagger\in\ran(A)$} and $\rho(t,u)$ solve the mean-field dynamics of stochastic EKI.
Then, {letting $u\in\R^n$ be an arbitrary element satisfying $Au=y^\dagger$,} the mean-field mean $\mathfrak{m}(t)$ satisfies:
\begin{align*}
    \mathfrak m(t) &= \mathfrak m_0 + \left(E-{\mathfrak C(t)\mathbf C_0^{-1}}\right)({u} - \mathfrak m_0)
     + t\mathfrak C(t)A^T\Gamma^{-1}\eps,\\
    \lim_{t\to\infty} \mathfrak m(t) &= \mathfrak m^{\dagger} + (A^T\Gamma^{-1}A)^-(A^T\Gamma^{-1}\eps),
\end{align*}
where $\mathfrak m^{\dagger} = \mathfrak m_0 + \left(E-{\mathfrak C_\infty \mathfrak C_0^{-1}}\right)({u} - \mathfrak m_0)$ is equivalently characterized by 
\begin{align*}
\mathfrak m^{\dagger} &= \argmin\left\lbrace\norm{\mathfrak m-\mathfrak m_0}_{\mathfrak C_0}\st \mathfrak m\in\R^n,\,A\mathfrak m = {y^\dagger}\right\rbrace.
\end{align*}
Alternatively, we can write 
\begin{align*}
    \lim_{t\to\infty}\mathfrak{m}(t) &=  \argmin\left\lbrace\norm{\mathfrak{m}-\mathfrak{m}_0}_{\mathfrak C_0} \st \mathfrak{m}\in \R^n,\,A\mathfrak{m} = \Pi_{\ran(A)}^\Gamma(y)\right\rbrace,
\end{align*}
where $\Pi_V^\Gamma$ is the $\Gamma$-orthogonal projection operator onto a closed subspace $V$.
The rate of convergence of $\mathfrak m(t)$ to $\mathfrak m_\infty$ is given by ${1}/{t}$.
\end{theorem}
\begin{remark}
It is not obvious but can be verified with a short calculation that formula for $\mathfrak m(t)$ constructed in this theorem indeed coincides with \labelcref{eq:sol_mean_mean_field} from \cite{garbuno2020interacting} for $\eps=0$.
In particular, it coincides with the posterior mean for $t=1$.
\end{remark}
\begin{remark}[asymptotic behavior]
\cref{thm:deterministic_EKI,thm:mean_field} show that the asymptotic behavior of all the different notions of mean as $t\to\infty$ coincide. 
Hence, in the case of vanishing noise ($\eps\approx 0$), where on would like to let the evolution proceed to large times $t$, both {naive} and mean-field EKI can be used comparably, {albeit mean-field EKI exhibits a significantly better convergence rate}.
\end{remark}

\subsection{Fundamental dynamical properties}\label{sec:char_ode}

{We can treat all equations arising in \cref{thm:deterministic_EKI,thm:mean_field}} in a uniform way by considering the ordinary differential equations
\begin{alignat}{2}
    \dot{\cov}(t) &= -\alpha\cov(t) A^T \Gamma^{-1} A \cov(t), \quad &&\cov(0) = \cov_0,\\
    \label{eq:ODE_general}
    \dot{x}(t) &= -\cov(t)A^T\Gamma^{-1}(Ax(t)-y),\quad &&x(0)=x_0.
\end{alignat}

In the following proposition we derive the solution of this ordinary differential equation and its asymptotic behavior. 
{For this we utilize the diagonalization of $\cov(t)$, which is provided by \cref{thm:cov_dynamics}.}

\begin{proposition}\label{prop:char_ODE}
The solution of \labelcref{eq:ODE_general} is given by
\begin{align}\label{eq:ODE_general_sol_no_source}
    x(t) = \left(\cov(t)\cov_0^{-1}\right)^\frac{1}{\alpha} x_0 + \left(\cov(t)\cov_0^{-1}\right)^\frac{1}{\alpha} \int_0^t \left(\cov(s)\cov_0^{-1}\right)^{1-\frac{1}{\alpha}}\d s\, \cov_0 A^T\Gamma^{-1} y.
\end{align}
\end{proposition}
\begin{proof}
We define $L(t) := \left(\cov(t)\cov_0^{-1}\right)^{-\frac{1}{\alpha}}=S E(t)^{-\frac{1}{\alpha}}S^{-1}$.
Using the definition of the matrices $E(t)$, $D$, and $D(t)$ as in \cref{thm:cov_dynamics}, we can compute
\begin{align*}
\dot{L}(t) &= S\diag\left((1+\alpha t\mu_i)^{\frac{1}{\alpha}-1}\mu_i\right)_{i=1}^nS^{-1}=SE(t)^{-\frac{1}{\alpha}}E(t)DS^{-1}\\
&=S E(t)^{-\frac{1}{\alpha}}D(t)S^{-1}\\
&=: S M_\alpha(t)S^{-1}.
\end{align*}
The product $L(t)x(t)$ then satisfies
\begin{align*}
    &\phantom{=}\frac{\de}{\de t}\left[L(t)x(t)\right]\\
    &=\dot{L}(t)x(t)+L(t)\dot{x}(t) \\
    &= S M_\alpha(t) S^{-1} x(t) - S E(t)^{-\frac{1}{\alpha}} S^{-1} \cov(t) A^T\Gamma^{-1}(Ax(t)-y)\\
    &= S M_\alpha(t) S^{-1} x(t) - S E(t)^{-\frac{1}{\alpha}} S^{-1} S D(t) S^{-1}x(t) + SE(t)^{-\frac{1}{\alpha}}S^{-1}\cov(t)A^T\Gamma^{-1}y\\
    &= S M_\alpha(t) S^{-1} x(t) - S \underbrace{E(t)^{-\frac{1}{\alpha}}D(t)}_{=M_\alpha(t)} S^{-1}x(t) + SE(t)^{-\frac{1}{\alpha}}S^{-1}\cov(t)A^T\Gamma^{-1}y \\
    &= SE(t)^{-\frac{1}{\alpha}}S^{-1}\cov(t) A^T\Gamma^{-1}y\\
    &=  SE(t)^{-\frac{1}{\alpha}}S^{-1}S E(t) S^{-1} \cov_0 A^T\Gamma^{-1}y\\
    &= SE(t)^{1-\frac{1}{\alpha}}S^{-1} \cov_0 A^T\Gamma^{-1}y\\
    &= \left(\cov(t)\cov_0^{-1}\right)^{1-\frac{1}{\alpha}} \cov_0 A^T\Gamma^{-1}y.
\end{align*}
We can integrate this equation to the following one, which is equivalent to \labelcref{eq:ODE_general_sol_no_source}:
\begin{align*}
    L(t)x(t) = x_0 + \int_0^t \left(\cov(s)\cov_0^{-1}\right)^{1-\frac{1}{\alpha}}\d s\, \cov_0 A^T\Gamma^{-1} y.
\end{align*}
\end{proof}

{
Let us now assume that the measured data is given by $y={y^\dagger} + \eps$ {with $y^\dagger\in\ran(A)$}.
Then we can split the integral in \labelcref{eq:ODE_general_sol_no_source} into two parts, where the first one can be evaluated using that ${y^\dagger}$ is in the range of the forward operator. 
\begin{proposition}
Let $y={y^\dagger} + \eps$ {with $y^\dagger\in\ran(A)$}. 
Then, {letting $u\in\R^n$ be an arbitrary element satisfying $Au=y^\dagger$,} the solution of \labelcref{eq:ODE_general} is given by
\begin{align}\label{eq:ODE_general_sol_source_and_noise}
    \begin{split}
        x(t) &= x_0 + \left(E-\left({\cov(t)\cov_0^{-1}}\right)^\frac{1}{\alpha}\right)({u} - x_0) \\
        &\qquad + \left({\cov(t)\cov_0^{-1}}\right)^\frac{1}{\alpha} \int_0^t \left({\cov(s)\cov_0^{-1}}\right)^{1-\frac{1}{\alpha}}\d s\,\cov_0A^T\Gamma^{-1}\eps.
    \end{split}
\end{align}
\end{proposition}
\begin{proof}
From \cref{prop:char_ODE} we know that $x(t)$ has the expression \labelcref{eq:ODE_general_sol_no_source}.
Plugging in $y={y^\dagger}+\eps$ we can split the integral into two. 
The second one, featuring $\eps$, just remains as it is.
The first one can be computed as follows, {using $y^\dagger=Au$ for some $u\in\R^n$}:
\begin{align*}
    &\phantom{=}\left({\cov(t)\cov_0^{-1}}\right)^\frac{1}{\alpha}\int_0^t \left({\cov(s)\cov_0^{-1}}\right)^{1-\frac{1}{\alpha}} \d s\,\cov_0A^T\Gamma^{-1}{y^\dagger}\\
    &= SE(t)^\frac{1}{\alpha}\int_0^t\diag\left({\mu_i}({1+\alpha s\mu_i})^{\frac{1}{\alpha}-1}\right)_{i=1}^n\d s\, S^{-1}{u}\\
    &= SE(t)^\frac{1}{\alpha}\diag\left(({1+\alpha t\mu_i})^\frac{1}{\alpha}-1\right)_{i=1}^n S^{-1}{u}\\
    &= S\left(E-E(t)^\frac{1}{\alpha}\right)S^{-1}{u}\\
    &= \left(E-\left({\cov(t)\cov_0^{-1}}\right)^\frac{1}{\alpha}\right){u}.
\end{align*}
Hence, using this together with \labelcref{eq:ODE_general_sol_no_source} we obtain the desired expression:
\begin{align*}
    x(t) &= \left(\cov(t)\cov_0^{-1}\right)^\frac{1}{\alpha}x_0 + \left(E-\left({\cov(t)\cov_0^{-1}}\right)^\frac{1}{\alpha}\right){u}\\
    &\qquad+ \left(\cov(t)\cov_0^{-1}\right)^\frac{1}{\alpha} \int_0^t \left(\cov(s)\cov_0^{-1}\right)^{1-\frac{1}{\alpha}}\d s\, \cov_0 A^T\Gamma^{-1} \eps\\
    &=x_0 + \left(E-\left({\cov(t)\cov_0^{-1}}\right)^\frac{1}{\alpha}\right)({u}-x_0)\\
    &\qquad+ \left(\cov(t)\cov_0^{-1}\right)^\frac{1}{\alpha} \int_0^t \left(\cov(s)\cov_0^{-1}\right)^{1-\frac{1}{\alpha}}\d s\, \cov_0 A^T\Gamma^{-1} \eps.
\end{align*}
\end{proof}}
\subsection{Asymptotic behavior and convergence rates}\label{sec:asymp_rates}
If one tries to solve the integral in \labelcref{eq:ODE_general_sol_no_source} or \labelcref{eq:ODE_general_sol_source_and_noise} one obtains
\begin{align*}
    &\int_0^t\left({\cov(s)\cov_0^{-1}}\right)^{1-\frac{1}{\alpha}}\d s \\
    &= S\diag\left(\frac{({1+\alpha t\mu_1})^\frac{1}{\alpha}-1}{\mu_1},\dots,\frac{({1+\alpha t\mu_k})^\frac{1}{\alpha}-1}{\mu_k},t,\dots,t\right)S^{-1}
\end{align*}
and therefore
\begin{align*}
    \left({\cov(t)\cov_0^{-1}}\right)^\frac{1}{\alpha}&\int_0^t\left({\cov(s)\cov_0^{-1}}\right)^{1-\frac{1}{\alpha}}\d s \\ &=
    S\diag\left(\frac{1-\frac{1}{(1+\alpha t\mu_1)^{\frac{1}{\alpha}}}}{\mu_1},\dots,\frac{1-\frac{1}{({{1+\alpha t\mu_k})^{\frac{1}{\alpha}}}}}{\mu_k},t,\dots,t\right)S^{-1}.
\end{align*}
The diagonal matrix in this expression blows up as $t\to\infty$ unless it is multiplied with $D=\diag(\mu_1,\dots,\mu_k,0,\dots,0)$. 
Remember that $D$ occured as diagonalization of $\cov_0A^T\Gamma^{-1}A=SDS^{-1}$. 
This shows that the integral term has to be multiplied with $\cov_0A^T\Gamma^{-1}A$ in order to exhibit a well-defined asymptotic behavior as $t\to\infty$.
{
Since in \labelcref{eq:ODE_general_sol_no_source,eq:ODE_general_sol_source_and_noise} the integral is multiplied by $\cov_0A^T\Gamma^{-1}y$ (or $\cov_0A^T\Gamma^{-1}\eps$ respectively), this shows that the data $y$ or the noise $\eps$ has to lie in $\ran(A)$ for a well-defined asymptotic behavior.

The previous observations can be collected in the following lemma.
\begin{lemma}\label{lem:intergral}
It holds that
\begin{align*}
    \lim_{t\to\infty}\left({\cov(t)\cov_0^{-1}}\right)^\frac{1}{\alpha}&\int_0^t\left({\cov(s)\cov_0^{-1}}\right)^{1-\frac{1}{\alpha}}\d s\,\cov_0A^T\Gamma^{-1}A = S(E-E_\infty)S^{-1}=E-\cov_\infty\cov_0^{-1}.
\end{align*}
\end{lemma}
Using this lemma we can study the asymptotic behavior of \labelcref{eq:ODE_general_sol_source_and_noise} under the assumption that $\eps\in\ran(A)$.
This can be assumed without loss of generality thanks to the following lemma, which assumes that the data is split into a range component and a component in the orthogonal complement 
\begin{align}
    \ran(A)^{\bot,\Gamma}:=\left\lbrace y \in \R^m\st \langle y, A u \rangle_\Gamma=0, \;\forall u\in\R^n \right\rbrace {=\ker(A^T\Gamma^{-1})}.
\end{align}
In our finite-dimensional observation setting this split is always possible since $\ran(A)$ is closed.

\begin{lemma}\label{lem:orthogonal_range}
Let $y= {y^\dagger} + y^\bot$ where ${y^\dagger} \in \ran(A)$ and $y^\bot \in \ran(A)^{\bot,\Gamma}$.
Then the ensemble Kalman inversion \labelcref{eq:enkf_stoch} with datum ${y^\dagger}$ coincides with the one for $y$.
\end{lemma}
\begin{proof}
The ensemble Kalman inversion \labelcref{eq:enkf_stoch} with datum $y$ reads
\begin{align*}
\dot{u}^j(t) &= -\cov(t) A^T\Gamma^{-1}(Au^j(t)-{y^\dagger} - y^\bot) +  C(t)A^T\Gamma^{-1}\sqrt{\Sigma}\dot{\wiener}^j(t) \\
&= -\cov(t) A^T\Gamma^{-1}(Au^j(t)-{y^\dagger}) + \cov(t) A^T\Gamma^{-1}y^\bot +  C(t)A^T\Gamma^{-1}\sqrt{\Sigma}\dot{\wiener}^j(t).
\end{align*}
Now the claim follows from
\begin{align*}
\cov(t)A^T\Gamma^{-1}y^\bot &= \frac{1}{J}\sum_{j=1}^J(u^j(t)-\mean(t))\langle u^j(t) - \mean(t), A^T\Gamma^{-1}y^\bot\rangle= 0.
\end{align*}
\end{proof}
Hence, we can formulate the following result on the asymptotic behavior.
As it can be expected from a regularization method for inverse problems, up to a noise term the time asymptotic limit \labelcref{eq:ODE_general_sol_source_and_noise} can be interpreted as projection of the initial datum $x_0$ onto the solution set of $\{Au=y^\dagger\}$ where $y^\dagger{\in\ran(A)}$.
In other words, $x(t)$ converges to a solution of $Au=y^\dagger$ with minimal value of the prior $x\mapsto\frac{1}{2}\norm{x-x_0}_{\cov_0}^2$, which can be interpreted as the formal limit of $t\to\infty$ of the variational regularization method \labelcref{eq:MAP}, see \cite{bungert2019solution} for a rigorous study of this phenomenon.
\begin{proposition}\label{prop:general_asymptotics}
Let {$y=y^\dagger + \eps$ where $y^\dagger\in\ran(A)$.} 
Then, {letting $u\in\R^n$ be an arbitrary element satisfying $Au=y^\dagger$,} the solution of \labelcref{eq:ODE_general} admits the asymptotic behavior
\begin{align*}
    \lim_{t\to\infty}x(t)  
    =x^\dagger + (A^T\Gamma^{-1}A)^-A^T\Gamma^{-1}\eps,
\end{align*}
where $x^{\dagger} := x_0 + \left(E-{\cov_\infty \cov_0^{-1}}\right)({u} - x_0)$ is equivalently characterized by 
\begin{align*}
x^{\dagger} = \argmin\left\lbrace\norm{x-x_0}_{\cov_0}\st x\in\R^n,\,Ax = {y^\dagger}\right\rbrace.
\end{align*}
Finally, we can also write
\begin{align*}
\lim_{t\to\infty}x(t) = \argmin\left\lbrace\norm{x-x_0}_{\cov_0}\st x\in\R^n,\,Ax = \Pi_{\ran(A)}^\Gamma(y)\right\rbrace.
\end{align*}
where $\Pi_V^\Gamma$ is the $\Gamma$-orthogonal projection operator onto a closed subspace $V$ {of $\R^m$}.
\end{proposition}
\begin{proof}
The proof is subdivided into four steps: First, we prove the asymptotic behavior, second, we show that $Ax^\dagger={y^\dagger}$, and third, we prove that $x^\dagger$ has minimal prior among all such parameters, after which we prove the alternative characterization of $\lim_{t\to\infty} x(t)$.\\
\textbf{Step 1:}
Using $\cov(t)\to\cov_\infty$ in \labelcref{eq:ODE_general_sol_source_and_noise}, applying \cref{lem:intergral}, and utilizing the fact that
\begin{align*}
    \left({\cov_\infty\cov_0^{-1}}\right)^\frac{1}{\alpha}=S {E_\infty}^\frac{1}{\alpha} S^{-1}=S E_\infty S^{-1} = \cov_\infty\cov_0^{-1}
\end{align*}
shows the identity
\begin{align*}
    \lim_{t\to\infty}x(t)=x_0+\left(E-\cov_\infty\cov_0^{-1}\right)({u}-x_0)+\left(E-\cov_\infty\cov_0^{-1}\right)\preim,
\end{align*}
where $\preim\in\R^n$ is such that $A\preim={\eps^\dagger}$, {where $\eps^\dagger:=\Pi_{\ran(A)}^\Gamma(\eps)$.}
Connecting this with the pseudo-inverse of $A^T\Gamma^{-1}A$ needs some thought:
Using $(A^T\Gamma^{-1}A)^- = SD^+S^{-1}\cov_0$ we obtain
\begin{align*}
    (A^T\Gamma^{-1}A)^-A^T\Gamma^{-1}A 
    &= S D^+ S^{-1}\cov_0 \cov_0^{-1} SDS^{-1}
    = SD^+DS^{-1} \\
    &= S(E-E_\infty)S^{-1} 
    = E - {\cov_\infty\cov_0^{-1}}.
\end{align*}
Finally, this implies
\begin{align*}
    \left(E-\cov_\infty\cov_0^{-1}\right)\preim &= 
    (A^T\Gamma^{-1}A)^-A^T\Gamma^{-1}A\preim\\
    &=(A^T\Gamma^{-1}A)^-A^T\Gamma^{-1}{\eps}^\dagger=(A^T\Gamma^{-1}A)^-A^T\Gamma^{-1}\eps,
\end{align*}
since $\eps={\eps^\dagger}+\eps^\bot$ where $\eps^\bot\in\ran(A)^{\bot,\Gamma}=\ker(A^T\Gamma^{-1})$.\\
\textbf{Step 2:}
Now we show that $Ax^\dagger = {y^\dagger}$.
To this end we rewrite $x^\dagger$ as
\begin{align*}
    x^\dagger = \cov_\infty\cov_0^{-1}(x_0-{u}) + {u} = SE_\infty S^{-1}(x_0-{u}) + {u}.
\end{align*}
Applying the matrix $A^T\Gamma^{-1}A$ to this equation yields
\begin{align*}
    A^T\Gamma^{-1}A x^\dagger &= A^T\Gamma^{-1}ASE_\infty S^{-1}(x_0-{u}) + A^T\Gamma^{-1}{y^\dagger}\\
    &= \cov_0^{-1}SDE_\infty S^{-1}(x_0 - {u}) + A^T\Gamma^{-1}{y^\dagger}\\
    &= A^T\Gamma^{-1}{y^\dagger}. 
\end{align*}
where we used that $D E_\infty=0$.
Since $\Gamma^{-1/2}A$ is trivially surjective on its range, the transpose $A^T\Gamma^{-1/2}$ is injective there.
This implies $\Gamma^{-1/2}A x^\dagger=\Gamma^{-1/2}{y^\dagger}$ and multiplication with $\Gamma^{1/2}$ shows that $A x^\dagger={y^\dagger}$.\\
\textbf{Step 3:}
We now show that $\norm{x^\dagger-x_0}_{\cov_0}$ is minimal among all parameters with $Ax={y^\dagger}$.
First, we note that $\cov_\infty\cov_0^{-1}=\cov_\infty\cov_0^{-1}\cov_\infty\cov_0^{-1}$ and hence $\cov_\infty = \cov_\infty\cov_0^{-1}\cov_\infty$.
Second, it holds that $x^\dagger-x_0=(E-\cov_\infty\cov_0^{-1})({u}-x_0)$.
We now claim that ${u}$ can be replaced by any $x\in\R^n$ such that $Ax={y^\dagger}$.
Indeed for $\preim={u}-x$ it holds $A\preim=0$ and hence
$0=\cov_0A^T\Gamma^{-1}A\preim=SDS^{-1}\preim$.
Since $S$ is invertible and $D=\diag(\mu_1,\dots,\mu_k,0,\dots,0)$, we infer that $S^{-1}\preim=(0,\dots,0,\ast,\dots,\ast)$ and therefore 
\begin{align*}
    \left(E-{\cov_\infty\cov_0^{-1}}\right)\preim = S(E-E_\infty)S^{-1}\preim = 0.
\end{align*}
This allows us to compute for any $x\in\R^n$ with $Ax={y^\dagger}$:
\begin{align*}
    \norm{x^\dagger-x_0}_{\cov_0}^2 &= \norm{(E-{\cov_\infty\cov_0^{-1}})(x-x_0)}_{\cov_0}^2 \\
    &= \norm{x-x_0}_{\cov_0}^2 - 2\left\langle x-x_0, {\cov_\infty\cov_0^{-1}}(x-x_0)\right\rangle_{\cov_0} + \norm{{\cov_\infty\cov_0^{-1}}(x-x_0)}_{\cov_0}^2\\
    &=\norm{x-x_0}_{\cov_0}^2 - 2\left\langle\cov_0^{-1}(x-x_0), {\cov_\infty\cov_0^{-1}}(x-x_0)\right\rangle + \norm{{\cov_\infty\cov_0^{-1}}(x-x_0)}_{\cov_0}^2.
\end{align*}
Using $\cov_\infty = \cov_\infty\cov_0^{-1}\cov_\infty$ and the symmetry of all covariance matrices, the inner product can be simplified as follows
\begin{align*}
    \left\langle\cov_0^{-1}(x-x_0), {\cov_\infty\cov_0^{-1}}(x-x_0)\right\rangle 
    &= \left\langle\cov_0^{-1}(x-x_0), \cov_\infty\cov_0^{-1}\cov_\infty\cov_0^{-1}(x-x_0)\right\rangle \\
    &= \left\langle\cov_0^{-1}\cov_\infty\cov_0^{-1}(x-x_0), \cov_\infty\cov_0^{-1}(x-x_0)\right\rangle \\
    &= \norm{\cov_\infty\cov_0^{-1}(x-x_0)}_{\cov_0}^2.
\end{align*}
Plugging this into the previous equation yields
\begin{align*}
    \norm{x^\dagger-x_0}_{\cov_0}^2 = \norm{x-x_0}_{\cov_0}^2-\norm{{\cov_\infty\cov_0^{-1}}(x-x_0)}_{\cov_0}^2\leq \norm{x-x_0}_{\cov_0}^2,
\end{align*}
which proves the statement since $Ax^\dagger = {y^\dagger}$ by Step 2 and $\norm{\cdot}_{\cov_0}$ is strictly convex.\\
\textbf{Step 4:} The last statement follows from everything we have proven so far by writing $y = {y^\dagger} + \eps^\dagger + \eps^\bot = \Pi_{\ran(A)}^\Gamma(y) + \eps^\bot$, hence
\begin{align*}
    \lim_{t\to\infty} x(t) &= \argmin\{\|x-x_0\|_{\cov_0}: x\in \R^n, Ax = \Pi_{\ran(A)}^\Gamma(y) \} + (A^T\Gamma^{-1}A)^-A^T\Gamma^{-1}\eps^\bot.
\end{align*}
Since $A^T\Gamma^{-1}\eps^\bot = 0$ by definition of the decomposition of $\eps$, we can conclude.
\end{proof}
}

{We close this section by deriving convergence rates of $x(t)$ to $x^\dagger$ in the case that there is no noise, i.e., $\eps=0$.
In the noisy case, \cref{prop:general_asymptotics} tells us that $x(t)$ does only converge to $x^\dagger$ up to a noise level.
Hence, in this case one can only expect a \emph{semi-convergence} behavior which we will not study in this article.
Instead, we report the following result for the noise-free case.}

\begin{proposition}\label{prop:x_cvgc_rates}
Assume that $y={y^\dagger}\in\ran(A)$ and let $x(t)$ and $x^\dagger$ be as in \cref{prop:general_asymptotics}.
Then there exists a constant $C>0$ such that it holds
\begin{align}
    \label{eq:rate_x}
    \norm{x(t)-x^\dagger} &\leq C\left(\frac{1}{\mathrm{gap}\cdot t}\right)^\frac{1}{\alpha},\quad\forall t\geq 0,\\
    \label{eq:rate_fwd_x}
    \norm{Ax(t)-y}_\Gamma &\leq C\cdot \left({\frac{\mu_{\max}}{t}}\right)^\frac{1}{\alpha},\quad \forall t \geq 0.
\end{align}
Here $\mathrm{gap}:=\mu_k$ denotes the spectral gap and $\mu_{\max}:=\mu_1$ the largest eigenvalue of the matrix $\cov_0A^T\Gamma^{-1}A$ .
\end{proposition}
\begin{proof}
Diagonalizing and subtracting $x(t)$ (given by \labelcref{eq:ODE_general_sol_source_and_noise} for $\eps=0$) and $x^\dagger$ (defined as in \cref{prop:general_asymptotics}) yields
\begin{align*}
    x(t) - x^\dagger &= S({E(t)}^\frac{1}{\alpha}-E_\infty)S^{-1}(x_0-\preim)\\
     &= S \diag\left(\frac{1}{({1+\alpha t\mu_1})^\frac{1}{\alpha}},\dots,\frac{1}{({1+2t\mu_k})^\frac{1}{\alpha}},0,\dots,0\right) S^{-1}(x_0-\preim),
\end{align*}
where $\mathrm{gap}:=\mu_k$ denotes the smallest non-zero eigenvalue of $A^T\Gamma^{-1}A$.
Taking norms yields the convergence rate \labelcref{eq:rate_x}.

Using this expression for $x(t)-x^\dagger$ and the diagonalization $A^T\Gamma^{-1}A = \cov_0^{-1} SDS^{-1}$ we also find
\begin{align*}
    &\phantom{=}A^T\Gamma^{-1}A(x(t)-x^\dagger)\\
    &=\cov_0^{-1}SD({E(t)}^\frac{1}{\alpha}-E_\infty)S^{-1}(x_0-\preim)\\ &=S \diag\left(\frac{\mu_1}{({1+\alpha t\mu_1})^\frac{1}{\alpha}},\dots,\frac{\mu_k}{({1+\alpha t\mu_k})^\frac{1}{\alpha}},0,\dots,0\right) S^{-1}(x_0-\preim).
\end{align*}
Multiplying this with $x(t)-x^\dagger$ and using $Ax^\dagger=y$ shows \labelcref{eq:rate_fwd_x}.
\end{proof}

\section{Ensemble and residual spreads of {naive} EKI for clean data}
\label{sec:spreads}

In this section we study the ensemble and residual spreads (defined in \labelcref{eq:ensemble_spread,eq:residual_spread}) of the {naive} EKI \labelcref{eq:enkf_stoch} with $\Sigma=0$.
{For this we utilize the system \labelcref{eq:ODE_cov_emp,eq:ODE_mean_emp} which describes the evolution of the {sample} covariance $C(t)$ and mean $m(t)$, see \cref{thm:deterministic_EKI}.}

\subsection{Convergence of the spreads}

{For convenience we repeat the definition of deviations $e^j$, residuals $r^j$, and residual mean $r$, given by 
\begin{align*}
    e^j(t) &:= u^j(t)-m(t),\\
    r^j(t) &:=u^j(t) - {u},\\
    r(t) &:= \frac{1}{J}\sum_{j=1}^J r^j(t) = m(t) - {u},
\end{align*}
and the time-dependent functions \labelcref{eq:ensemble_spread,eq:residual_spread,eq:fwd_ensemble_spread,eq:fwd_residual_spread}, which describe the ensemble collapse and residual convergence:
\begin{alignat*}{2}
    &V_e(t) = \frac{1}{2J}\sum_{j=1}^J \norm{e^j(t)}^2,
    &&V_r(t) = \frac{1}{2J}\sum_{j=1}^J \norm{r^j(t)}^2,\\
    &\mathfrak V_e(t) = \frac{1}{2J}\sum_{j=1}^J \norm{Ae^j(t)}_\Gamma^2,\quad
    &&\mathfrak V_r(t) = \frac{1}{2J}\sum_{j=1}^J \norm{Ar^j(t)}_\Gamma^2.
\end{alignat*}}
We recall from \cite{blomker2019well} that the ensemble spread $V_e(t)$ decreases monotonously with time.
It does not necessarily converge to zero unless in the fully observed case of an invertible forward operator $A$, as we can only expect ensemble collapse along the components orthogonal to the kernel of $A$. 

To see the decrease of $V_e$ one computes
\begin{align*}
    \dot{V}_e(t) &= \frac{1}{J}\sum_{j=1}^J \langle \dot e^j(t), e^j(t)\rangle = -\frac{1}{J}\sum_{j=1}^J \langle C(t)A^T\Gamma^{-1}A e^j(t), e^j(t)\rangle\\
    &= -\frac{1}{J^2}\sum_{i,j=1}^J \langle e^i(t), e^j(t)\rangle \langle e^i(t), A^T\Gamma^{-1}A e^j(t)\rangle\leq 0,
\end{align*}
where Lemma~A.3. in~\cite{blomker2019well} ensures that the last term is non-negative.
For the ensemble spread in observation space, given by \labelcref{eq:fwd_ensemble_spread}, one can even prove convergence to zero. 
It holds
\begin{align*}
    \dot{\mathfrak V}_e(t) &= \frac{1}{J}\sum_{j=1}^J\langle A^T\Gamma^{-1}Ae^j(t),\dot e^j(t)\rangle = - \frac{1}{J}\sum_{j=1}^J\langle A^T\Gamma^{-1}Ae^j(t),C(t)A^T\Gamma^{-1}A e^j(t)\rangle \\
    &=-\frac{1}{J^2}\sum_{j,k=1}^J \langle A^T\Gamma^{-1} Ae^j(t),e^k(t)\rangle^2 = -\frac{1}{J^2}\sum_{j,k=1}^J\langle Ae^ j(t),Ae^k(t)\rangle_\Gamma^2.
\end{align*}
We can bound this by removing all terms $j\neq k$ and applying Jensen's inequality for sums for the convex function $x\mapsto x^2$.
\begin{align*}
    \dot{\mathfrak V}_e(t) &\leq -\frac{1}{J^2}\sum_{j=1}^J\norm{Ae^ j(t)}_\Gamma^4 \leq -\frac{1}{J}\left(\frac{1}{J}\sum_{j=1}^J\norm{Ae^ j(t)}_\Gamma^2\right)^2 = -\frac{4}{J}\mathfrak V_e(t)^2.
\end{align*}
An ordinary differential equation comparison principle then yields that 
\begin{align*}
\mathfrak V_e(t) \leq \frac{1}{\frac{4}{J} t + \frac{1}{\mathfrak V_e(0)}}.    
\end{align*}
Monotonous decrease is not true for the residual spread, whose derivative in time is
\begin{align*}
    \dot{V}_r(t) &= \frac{1}{2J}\sum_{j=1}^J \langle \dot r^j(t), r^j(t)\rangle = -\frac{1}{2J}\sum_{j=1}^J \langle A^T\Gamma^{-1}A C(t) r^j(t), r^j(t)\rangle\\
    &= -\frac{1}{2J^2}\sum_{i,j=1}^J \langle e^i(t), r^j(t)\rangle \langle e^i(t), A^T\Gamma^{-1}A r^j(t)\rangle,
\end{align*}
which does not carry a sign in general.
{The residual spread does indeed not decay mo\-not\-o\-nous\-ly for {naive} EKI with noise-free data, which we study in the following.}
Indeed, simple simulations (see \cref{fig:nonmonotonicity}) show that the residual norms can increase in time. 
There are two issues at play here:
\begin{itemize}
    \item Correct choice of ${u}$ in $r^j(t)=u^j(t)-{u}$ and
    \item Skewness of the Euclidean norm with respect to the EKI dynamics.
\end{itemize}
First, the residuals $r^j$ and their mean $r$ are defined via a choice of  {``best guess''} ground truth parameter ${u}$ such that $A{u} = {y^\dagger}$, {where $y^\dagger\in\ran(A)$}.
While the degrees of freedom in this choice are irrelevant in the observation space, they are eminent in the parameter domain: 
For a given initial ensemble $\{u_0^j\}_{j=1}^J$, the mean $m(t)$ will converge to a well-defined limit $m_\infty := \lim_{t\to\infty}m(t)$ according to \cref{thm:deterministic_EKI} {and in the noise-free case it holds $m_\infty=m^\dagger$}.
It is exactly this parameter we need to choose as a candidate for the reference parameter ${u}$ in $r^j(t) := u^j(t)-{u}$, which can be seen from the following reformulation of $V_r$.
\begin{proposition}
Let ${u}\in\R^n$ such that $y=A{u}$ and {let the residuals be defined by $r^j(t)=u^j(t)-{u}$}.
It holds 
\begin{align}\label{eq:reform_res_spread}
    V_r(t) = V_e(t) + \frac{1}{2}\norm{m(t)-{u}}^2. 
\end{align}
\end{proposition}
\begin{proof}
This can be seen as follows:
\begin{align*}
    V_r(t) &= \frac{1}{2J}\sum_{j=1}^J\norm{r^j(t)}^2 
    =\frac{1}{2J}\sum_{j=1}^J\left(\norm{u^j(t)}^2-2\langle u^j(t),{u}\rangle+\norm{{u}}^2\right) \\
    &=\frac{1}{2J}\sum_{j=1}^J\Big(\norm{u^j(t)}^2-2\langle u^j(t),m(t)\rangle+\norm{m(t)}^2 -2\langle u^j(t),{u}-m(t)\rangle\\
    &\hspace{5cm}+\norm{{u}}^2-\norm{m(t)}^2\Big) \\
    &=\frac{1}{2J}\sum_{j=1}^J\norm{u^j(t)-m(t)}^2 + \frac{1}{2}\norm{{u}}^2-\langle m(t),{u}-m(t)\rangle - \frac{1}{2}\norm{m(t)}^2 \\
    &=V_e(t) + \frac{1}{2}\norm{m(t)-{u}}^2. 
\end{align*}
\end{proof}

{We can see that ${u} := m^\dagger$, defined as in \cref{thm:deterministic_EKI},} is the canonical choice because it guarantees {that the second term in the decomposition of $V_r(t)$ vanishes.}
We obtain the following trivial corollary that for ${u}=m^\dagger$ the ensemble spread and the residual spread converge to the same value (which is not zero, in general).
\begin{corollary}
Let ${u}:=m^\dagger$. 
Then it holds
\begin{align}
    \lim_{t\to\infty}V_e(t) = \lim_{t\to\infty}V_r(t).
\end{align}
\end{corollary}
\begin{proof}
The proof follows from \labelcref{eq:reform_res_spread} together with the fact that $m(t)\to m^\dagger$ {in the noiseless case according to \cref{thm:deterministic_EKI}.}
\end{proof}

However, even with this choice of ${u=m^\dagger}$, the residual spread $V_r$ can still \emph{fail to decrease} (see \cref{fig:nonmonotonicity} for an example). 
In contrast, the residual spread in observation space \labelcref{eq:fwd_residual_spread} does indeed decrease monotonously.
Similar to above one can express $\mathfrak V_r(t)$ in terms of $\mathfrak V_e(t)$ and obtain monotonous convergence to zero.
\begin{proposition}
The residual spread in parameter space $\mathfrak V_r(t)$ admits the expression
\begin{align}\label{eq:reform_fwd_res_spread}
    \mathfrak V_r(t) = \mathfrak V_e(t) + \frac{1}{2}\norm{A m(t) - y}_\Gamma^2.
\end{align}
Furthermore $t\mapsto\mathfrak V_r(t)$ is non-increasing and converges to zero with rate $1/t$.
\end{proposition}
\begin{proof}
The proof of \labelcref{eq:reform_fwd_res_spread} work precisely as above for $V_r(t)$.
Furthermore, it holds
{%
\begin{align*}
    \frac{\d}{\d t}\frac1 2 \|A m(t)-y\|_\Gamma^2 
    &= \langle A^T\Gamma^{-1}(A\mean(t)-y),\dot{m}(t)\rangle\\ 
    &=-\langle A^T\Gamma^{-1}(A m(t)-y),C(t)A^T\Gamma^{-1}(A m(t)-y)\rangle\\
    &=-\frac{1}{J}\sum_{j=1}^J\langle A^T\Gamma^{-1}(A m(t)-y),e^j(t)\otimes e^j(t)A^T\Gamma^{-1}(A\mean(t)-y)\rangle\\
    &=-\frac{1}{J}\sum_{j=1}^J\langle A^T\Gamma^{-1}(A m(t)-y),e^j(t)\rangle^2
    \leq 0.
\end{align*}
}
Hence, using that $t\mapsto\mathfrak V_e(t)$ converges to zero monotonously with rate $1/t$ and that the same holds true for $t\mapsto\frac1 2 \|A m(t)-y\|_\Gamma^2$ (see~\labelcref{eq:rate_fwd_x} in the case $\alpha=2$ for the rate), we obtain the assertion.
\end{proof}

We emphasize again that monotone convergence of $t\mapsto \frac{1}{2}\|A m(t)-y\|_\Gamma^2$ does not mean that the quantities $\norm{m(t)-m^\dagger}$ or $V_r(t)$ decrease as well. 
First, the mapping of this quantity via $A$ only keeps track of the data-informed parameter dimensions, i.e., those orthogonal to the kernel of $A$. And secondly, even invertibility of $A$ still does not imply monotonicity of $\|m(t)-{u}\| $ as the mapping $A$ can warp the coordinate system in such a way that this property is lost. 
This can be seen in an elementary example unrelated to the EKI: Consider the curve $x(t) = (\cos(t),\sin(t))$ for which $V(t) := \|x(t)\|^2$ is constant, i.e., monotonously non-increasing. 
On the other hand, with $A=\diag(2,1)$, the mapping $\tilde V(t) = \|Ax(t)\|^2$ is not monotonous. 

\begin{example}\label{ex:nonmon}
As a concrete example for the non-monotonicity of the mean and the residual, we can consider the forward operator $A = \diag(100,1)$, observation $y = (0,0)^T$, and an initial ensemble with mean $m_0 = (100,100)^T$ and {sample} covariance 
\begin{align*}
C_0 = 
\begin{pmatrix}
\phantom{-}25 & -24\\
-24 & \phantom{-}25
\end{pmatrix}, 
\end{align*}
whose eigenvectors are $(-1,1)^T$ and $(1,1)^T$ with eigenvalues $49$ and $1$, respectively.

\cref{fig:nonmonotonicity} shows the initial ensemble and the trajectories of the ensemble and its sample mean in the parameter space.
Clearly, the sample mean and the whole ensemble move away from their final limit $(0,0)^T$ for quite some time until they finally ``change direction'' and converge towards their limit. 
The initial shearing of the ensemble combined with the strong weighting of the horizontal direction, which is encoded in the forward operator, leads to an initial movement of the ensemble along its principal axis to the top left.
\begin{figure}[hbt]
    \centering
    {\includegraphics[width=0.6\textwidth,trim=2cm 2.5cm 5cm 2.5cm,clip]{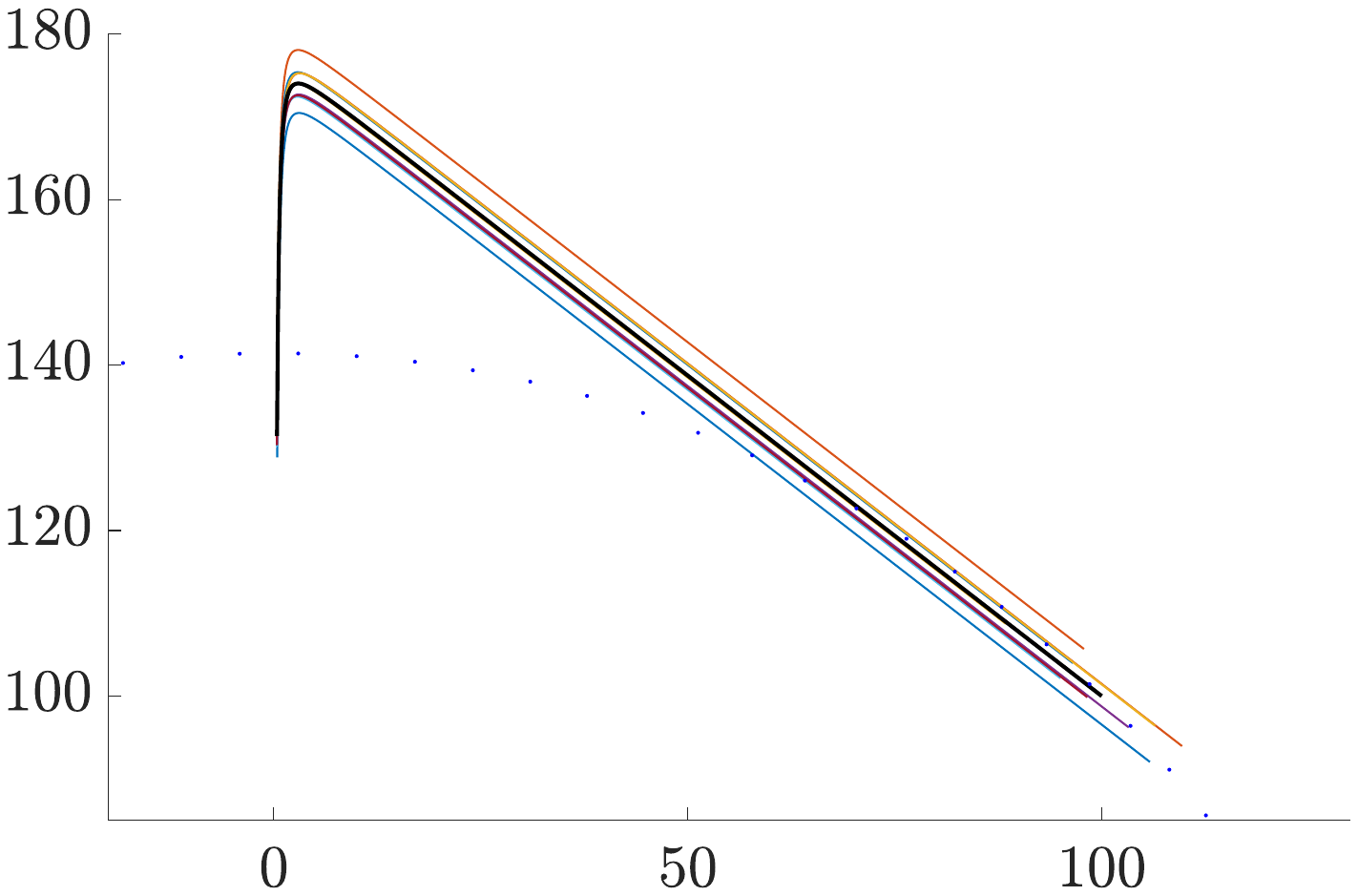}}%
    \caption{From bottom right to top left: Trajectories of {naive} EKI (black curve is the mean $m(t)$) for $t\in[0,1]$. The Euclidean sphere (dotted) through $m_\infty$ demonstrates non-monotonicity of the mean.}
    \label{fig:nonmonotonicity}
\end{figure}
\end{example}

Applied to the present setting, this means that the Euclidean norm is not the natural norm with respect to which we should view the dynamics of the ensemble.
Hence, we need to either settle for \textit{non-monotonous convergence} of $\norm{m(t)-m^\dagger}$, or we need to pick a more \textit{problem-adapted norm}, as presented in the following.

\subsection{Monotonicity in a problem-adapted norm}
Now we see how to define a new norm on the parameter space with respect to which we can prove monotonous convergence of the residuals {of {naive} EKI with data in the range of the forward operator}. 

Recall the ordinary differential equation \labelcref{eq:ODE_mean_emp} for {the {sample} mean $m(t)$ with data} $y=A\preim\in\ran(A)$:
\begin{align*}
    \dot m(t) = -C(t)A^T\Gamma^{-1}(Am(t)-y) = -SD(t)S^{-1}(m(t) - \preim).
\end{align*}
By defining $\rho(t) = S^{-1}m(t)$, we obtain the ordinary differential equation
\begin{align*}
    \dot\rho(t) = -D(t) (\rho(t) - S^{-1}\preim),
\end{align*}
which decouples into $n$ ordinary differential equations since $D(t)$ is a diagonal matrix. 
This allows us to prove the following Lyapunov type estimate.
\begin{proposition}\label{prop:lyapunov}
Let $S$ be such that $C_0 A^T\Gamma^{-1}A = SDS^{-1}$ as in \cref{thm:cov_dynamics}. Then 
\begin{align*}
     L(m) := \frac{1}{2}\|S^{-1}(m - \preim)\|^2
\end{align*}
is a Lyapunov function for the dynamics of the {sample} mean $m(t)$, meaning that $\frac{\de}{\de t}{L}(m(t)) \leq 0$.
\end{proposition}
\begin{proof}
We observe $L(m(t))=\frac{1}{2}\norm{\rho(t) - S^{-1}\preim}^2$ and compute
\begin{align*}
    \frac{\de}{\de t}L(m(t)) = \langle \dot\rho(t), \rho(t)-S^{-1}\preim\rangle = -\norm{\sqrt{D(t)}\rho(t)}^2\leq 0.
\end{align*}
\end{proof}

\begin{remark}
A couple of remarks regarding this Lypunov approach are in order.
\begin{itemize}
    \item The strength of using the norm $\norm{S^{-1}\cdot}$ instead of the Euclidean norm, is that it captures exactly the correct notion of convergence of $x$ by respecting both the influence of the forward mapping $A$ and the initial ensemble $C_0$ {(which is either implicitly assumed to be an approximation of the prior, or explicitly chosen to have similar statistical properties). In particular, the ``pure'' Euclidean norm does not yield the right notion of tracking convergence, as \cref{ex:nonmon,fig:nonmonotonicity} show.}
    \item The proof that the dynamics behave monotonously in this norm did not use the explicit solution for $m(t)$ derived in \cref{thm:deterministic_EKI}.
    Therefore, we hope that this Lyapunov approach might work similarly in the stochastic setting where an explicit solution is not readily available.
    \item Since $S^{-1}$ is a regular matrix, the functional $L(m)$ is coercive which will turn out useful for showing existence of a limit of $m(t)$ as $t\to\infty$ in more general settings.
    \item ``Preconditioning'' with $S^{-1}$ also allows one to show that $C(t)A^T\Gamma^{-1}A$ has the same eigenvectors for all times, without using the explicit solution for $C(t)$.
    From the identity $C_0 A^T\Gamma^{-1}A = SDS^{-1}$ we see that $\dot{C}(t)$ is diagonalized in the same way:
    \begin{align*}
        S^{-1} \dot{C}(t)A^T\Gamma^{-1}A S = -S^{-1} C(t)A^T\Gamma^{-1}A C(t)A^T\Gamma^{-1}AS,
    \end{align*}
    i.e., if we set $D(t):=S^{-1}C(t)A^T\Gamma^{-1}AS$, we obtain the very simple ordinary differential equation $\dot{D}(t) = -D^2(t)$.
    This proves that $D(t)$ stays diagonal for all $t\geq 0$ and we obtain the diagonalization $C(t)A^T\Gamma^{-1}A = S D(t) S^{-1}$, which we have already derived in \cref{thm:cov_dynamics} with other techniques.
\end{itemize}
\end{remark}

\section{Spectral decomposition of the covariance}
\label{sec:spectral}

Recall that we are analysing the behavior of the EKI, with its covariance dynamics given by \labelcref{eq:ODE_cov}, which is
\begin{align*}
\dot{\cov}(t) = -\alpha\cov(t) A^T \Gamma^{-1} A \cov(t), \quad \cov(0) = \cov_0.    
\end{align*}
In the previous sections we have intensively used the diagonalization $\cov(t)\cov_0^{-1} = S E(t) S^{-1}$ from \cref{thm:cov_dynamics} to understand the {deterministic} dynamics of {different versions of} EKI. The eigenvectors of the matrix $\cov(t)\cov_0^{-1}$ do not change in time and their associated eigenvalues have an explicit expression, {as has been shown in \cref{sec:diagonalization}, and also in \cite{garbuno2020interacting}}. 
None of this is true for the covariance matrix $\cov(t)$ itself and, in particular, its eigenvectors can change drastically in time. 

In this section, we derive a coupled system of ordinary differential equations which is solved by the eigenvalue and eigenvectors of $\cov(t)$.
{For this we \emph{do not} use the explicit solution for $\cov(t)$. 
Instead we only utilize its governing dynamics \labelcref{eq:ODE_cov}, re-displayed above.}
To this end, denote by $\lambda_1(t),\dots,\lambda_n(t)$ the eigenvalues of $\cov(t)$, with eigenvectors $v_1(t),\ldots,v_n(t)$. 
Since $\cov(t)$ is symmetric, all eigenvectors can be chosen orthonormal and we can assume the ordering $\lambda_1(t) \geq \dots \geq \lambda_n(t)$ of the eigenvalues.

\begin{theorem}[eigenvector dynamics]\label{thm:eigenvectors}
Let $\cov(t)$ denote the solution of \labelcref{eq:ODE_cov} with initial condition $\cov_0$.
Denote the eigenvalues and normalized eigenvectors of $\cov(t)$ by $\lambda_i(t)$ and $v_i(t)$ for $i=1,\dots,n$.
Then it holds:
\begin{itemize}
    \item Any set of eigenvalues $\lambda_i(t)$ with corresponding eigenvectors $v_i(t)$ for $i=1,\dots,n$, differentiable in time, solve the following differential algebraic system:
        \begin{subequations}\label{eq:DAE}
        \begin{align}
            \label{eq:ODE_eigvalues}
            \dot{\lambda}_i &= -{\alpha}\lambda_i^2 \norm{A v_i}_\Gamma^2,\\
            \label{eq:ODE_eigvectors}
            \dot{v}_i &= \sum_{\substack{j\in\{1,\dots,n\}:\\\lambda_j \neq \lambda_i}} \frac{{\alpha}\lambda_i\lambda_j}{\lambda_j-\lambda_i}  \langle A v_i,  A v_j\rangle_\Gamma v_j,\\
            \label{eq:AE_orth}
            0 &= \lambda_i^2 \langle Av_i, Av_j\rangle_\Gamma,\quad\text{if }i\neq j\text{ but }\lambda_i=\lambda_j.
        \end{align}
        \end{subequations}
    \item  The eigenvalue adhere to the following bounds
        \begin{align}
        \lambda_i(t) &\geq \frac{\lambda_i(0)}{{\alpha}\norm{\Gamma^{-1/2}A}^2t\lambda_i(0) + 1}, \qquad\forall i=1,\dots,n,
        \\
        \lambda_{1}(t) &\geq \frac{\lambda_1(0)}{{\alpha}\norm{Av_1(0)}_\Gamma^2 t\lambda_1(0) + 1},
        \\
        \lambda_{n}(t) &\leq \frac{\lambda_n(0)}{{\alpha}\norm{Av_n(0)}_\Gamma^2t\lambda_n(0) + 1}.
        \end{align}
    \item The eigenvalues and eigenvectors have the following asymptotic behavior
        \begin{align}
            (\forall i = 1,\dots,n) \quad \lim_{t\to\infty}\lambda_i(t) = 0 \quad\text{or}\quad \lim_{t\to\infty} Av_i(t) = 0.
        \end{align}
\end{itemize}
\end{theorem}
\begin{remark}
The case of eigenvalues with $\lambda_i(t^\star) = \lambda_j(t^\star)$ but $\lambda_i(t)\neq \lambda_j(t)$ for $t\neq t^\star$ is a source of ambiguity for both the labelling of the eigenvalues and for the well-posedness of the system \labelcref{eq:DAE}. 
We disregard this case by considering the differential equations only for $t<t^\star$ and $t>t^\star$, and then completing by continuity. If $\lambda_i(t) = \lambda_j$ for all $t$ in a proper interval, then the dynamics of the eigenvectors has an intrinsic ambiguity which we remove by setting $\langle \dot v_i, v_j\rangle = 0$. 
\end{remark}
\begin{proof}
From the explicit solution \labelcref{eq:sol_cov} it follows that $\cov(t)$ is symmetric and positive definite for all times and thus diagonalizable with orthonormal eigenvectors. We start with the defining equations for the eigenvectors and eigenvalues while enforcing normality:
\begin{align*}
    \left(\cov(t)-\lambda_i(t)E\right)v_i(t) &= 0,\\
    \|v_i(t)\|^2 &= 1. 
\end{align*}
By taking the time derivative in both equations (and dropping the explicit dependence on $t$ for brevity), we obtain
\begin{align}
    (\cov-\lambda_i E)\dot{v}_i &= \dot{\lambda}_i v_i - \dot{\cov}v_i,\\
    \langle\dot{v}_i,v_i\rangle &=0. \label{eq:increment_orth}
\end{align}
By using \labelcref{eq:ODE_cov} and $\cov v_i = \lambda_i v_i$, this means that 
\begin{align*}
    \cov \dot{v}_i - \lambda_i\dot{v}_i 
    = \dot{\lambda}_i v_i + {\alpha}\cov A^T\Gamma^{-1}A\cov v_i
    = \dot{\lambda}_i v_i + {\alpha}\lambda_i \cov A^T\Gamma^{-1}A v_i.
\end{align*}
Now we take the scalar product of both sides with $v_j$ and obtain
\begin{align*}
    \langle\cov \dot{v}_i,v_j\rangle - \lambda_i \langle\dot{v}_i,v_j\rangle 
    = \dot{\lambda}_i \langle v_i,v_j\rangle + {\alpha}\lambda_i \langle\cov A^T\Gamma^{-1}A v_i,v_j\rangle,
\end{align*}
which is equivalent to
\begin{align}\label{eq:evecs_evals}
    (\lambda_j - \lambda_i)\langle\dot{v}_i,v_j\rangle 
    =  \dot{\lambda}_i \delta_{i,j} +  {\alpha}\lambda_i \lambda_j \langle A v_i,A v_j\rangle_\Gamma,
\end{align}
where $\delta_{i,j}$ is $1$ if $i=j$ and $0$ otherwise. 

We have three cases to consider: 
Firstly, in the case $i=j$, we get
\begin{align*}
    \dot{\lambda}_i = -{\alpha}\lambda_i^2 \|\Gamma^{-1/2}Av_i\|^2 = -2\lambda_i^2 \norm{Av_i}_\Gamma^2,
\end{align*}
which proves~\labelcref{eq:ODE_eigvalues}.

Secondly, if $i\neq j$ and $v_i, v_j$ are two different eigenvectors for the same eigenvalue $\lambda_i = \lambda_j$, then \labelcref{eq:evecs_evals} implies 
\begin{align*}
    \lambda_i^2 \langle A v_i, A v_j\rangle_\Gamma = 0,
\end{align*}
which proves \labelcref{eq:AE_orth}.

Thirdly, in the case $i\neq j$ and $\lambda_i\neq \lambda_j$ we conclude from \labelcref{eq:evecs_evals} that
\begin{align*}
   \langle \dot{v}_i,v_j\rangle = \frac{{\alpha}\lambda_i\lambda_j}{\lambda_j-\lambda_i}  \langle Av_i,Av_j\rangle_\Gamma.
\end{align*}

Using this equation together with \labelcref{eq:increment_orth} and orthonormality of the $v_i$, we can express $\dot{v}_i$ in the basis $\{v_j\}_j$ as
\begin{align*}
    \dot{v}_i = \sum_{\lambda_j \neq \lambda_i} \frac{{\alpha}\lambda_i\lambda_j}{\lambda_j-\lambda_i}  \langle A v_i,A v_j\rangle_\Gamma v_j.
\end{align*}
where we set $\langle\dot v_i, v_j\rangle = 0$ for $v_i,v_j$ in the same eigenspace for a joint eigenvalue $\lambda = \lambda_i=\lambda_j$.

The lower bound decay rate on the eigenvalues follows from the fact that the $v_i$ have unit norm and we can bound $\norm{A v_i}_\Gamma^2 \leq \norm{\Gamma^{-1/2}A}^2$. Furthermore, we see that 
\begin{align*}
    \Gamma^{-1/2}A \dot{v}_i = \sum_{\lambda_j\neq\lambda_i} \frac{{\alpha}\lambda_i\lambda_j}{\lambda_j-\lambda_i}  \langle A v_i,A v_j\rangle_\Gamma \Gamma^{-1/2}A v_j
\end{align*}
and thus 
\begin{align*}
    \frac{\d}{\d t}\frac{1}{2}\|A v_i\|_\Gamma^2 = \sum_{\lambda_j\neq \lambda_i} \frac{{\alpha}\lambda_i\lambda_j}{\lambda_j-\lambda_i}  \langle A v_i,A v_j\rangle_\Gamma^2,
\end{align*}
from which we can similarly derive bounds for the cases $i=1$, i.e., $\frac{\d}{\d t}\|A v_i\|_\Gamma^2 \leq 0$ and $i=n$, i.e., $\frac{\d}{\d t}\|A v_i\|_\Gamma^2 \geq 0$.

For the last pillar of \cref{thm:eigenvectors} we argue as follows:
From \cref{thm:cov_dynamics}, we know that
\begin{align*}
    A^T\Gamma^{-1}A\cov (t) = S\diag\left(\frac{\mu_i}{1+{\alpha} t\mu_i}\right)_{i=1}^n S^{-1} \to 0,\quad t\to\infty.
\end{align*}
If we now multiply this from the right with an eigenvector $v_i(t)$ of $\cov (t)$ corresponding to an eigenvalue $\lambda_i(t)$, we obtain
\begin{align*}
    \lambda_i(t) A^T\Gamma^{-1}A v_i(t) = A^T\Gamma^{-1}A\cov (t)v_i(t) \to 0,\quad t\to\infty.
\end{align*}
i.e., $\lambda_i(t)\to 0$ or $Av_i(t)\to 0$, as claimed.
\end{proof}

\begin{corollary}
The function $t\mapsto \lambda_1(t)$ is convex.
\end{corollary}
\begin{proof}
We can take another derivative in \labelcref{eq:ODE_eigvalues} and obtain
\begin{align*}
    \ddot{\lambda}_i 
    &= -2{\alpha}\lambda_i\dot{\lambda}_i\norm{Av_i}_\Gamma^2 - 2{\alpha}\lambda_i^2\frac{\de}{\de t}\frac{1}{2}\norm{Av_i}_\Gamma^2 \\
    &= 2{\alpha}^2 \lambda_i^3\norm{Av_i}_\Gamma^4 - 2{\alpha}^2\lambda_i^3 \sum_{\lambda_j\neq\lambda_i}\frac{\lambda_j}{\lambda_j-\lambda_i}\langle Av_i,Av_j\rangle_\Gamma^2\\
    &=2{\alpha}^2\lambda_i^3\left[\norm{Av_i}_\Gamma^4 - \sum_{\lambda_j\neq\lambda_i}\frac{\lambda_j}{\lambda_j-\lambda_i}\langle Av_i,Av_j\rangle_\Gamma^2\right].
\end{align*}
Hence, for $i=1$ the second term is non-positive and one obtains $\ddot{\lambda}_i \geq 0$ which implies convexity.
\end{proof}

\section{Conclusions and outlook}

In this article we have provided a complete description of the deterministic dynamics of the ensemble Kalman Inversion, based on the spectral decomposition of a preconditioned {sample} covariance operator. 
We focused on {naive} EKI and mean-field EKI.
In particular, we have derived the time-asymptotic behavior of the covariance and also studied asymptotic profiles, their second-order asymptotics.
Then, we computed the explicit dynamics and convergence rates of particles and their {sample} mean for noisy data, and showed consistency for vanishing noise.
We applied these findings to the study of ensemble and residual spreads of {naive} EKI both in observation and parameter space, in particular, we constructed a counter example which shows that the residual spread in the Euclidean norm is not decreasing, in general.
This inspired us to define a ``problem-adapted'' weighted norm, depending on the forward model and the prior, with respect to which the ensemble spread decreases monotonously.
We concluded our studies with a spectral analysis of the ``pure'' covariance operator and derived the governing differential equations for its eigenvalues and eigenvectors.

Our analysis shows that if one is interested in the Bayesian context, the stochastic version or an approximation of mean-field EKI might {be} a better choice than {naive} EKI.
{In contrast, in the viewpoint of EKI as a derivative-free optimization method (e.g., in image reconstruction where the exact noise model is unknown), this is less relevant.
There, the noise covariance $\Gamma$ is typically guessed or treated as a regularization parameter and therefore the transport of prior to posterior in unit time looses its meaning.
In this context, EKI can be viewed as (stochastic) gradient descent of the functional $u\mapsto\frac{1}{2}\norm{Au-y}_\Gamma^2$ with respect to a covariance weighted norm, see also the discussion in the introduction.
}
All versions considered here share the same asymptotic behavior, albeit with different rates of convergence, and therefore are all basically equivalent in the zero-noise limit.
The presented results about the convergence rates might be used to accelerate EKI using time re-parametrization.
Furthermore, they might help to derive sharp quantitative estimates for the number of particles $J\in\N$ needed based on the noise level.
{Our methods show a way of analyzing the dynamics of certain particle systems related to the EKI. An interesting avenue for further research is the application of these ideas to generalized and improved variants of (Tikhonov-regularized) EKI, for example Ensemble Square Root filters \cite{tippett2003ensemble,chada2022convergence} and Ensemble Transform Kalman filters \cite{bishop2001adaptive} as applied to inverse problems.}
{
An important extension of our results is the analysis of the averaged stochastic equations \labelcref{eq:ODE_mean_av_emp,eq:ODE_cov_av_emp}.
This is non-trivial since it requires a careful control of the covariance-like corrections terms and their behavior in terms of the ensemble size~$J$.
}
Even more general, one can consider the much more challenging settings of purely stochastic EKI (i.e., where we need to analyse stochastic differential equations governing the evolution of the ensemble members), the EKI applied to a nonlinear forward problem (which means that the Bayesian posterior is not a Gaussian measure anymore), and of time-discretizations of the EKI used on problems in practice.

\section*{Acknowledgement}
The authors would like to thank the referees for their helpful remarks and guidance.
LB acknowledges funding by the Deutsche Forschungsgemeinschaft (DFG, German Research Foundation) under Germany's Excellence Strategy - GZ 2047/1, Projekt-ID 390685813. PW acknowledges support from MATH+ project EF1-19: Machine Learning Enhanced Filtering Methods for Inverse Problems, funded by the Deutsche Forschungsgemeinschaft (DFG, German Research Foundation) under Germany's Excellence Strategy – The Berlin Mathematics Research Center MATH+ (EXC-2046/1, project ID: 390685689).

\printbibliography

\clearpage
\begin{appendix}
\section{Auxiliary derivations}
\begin{lemma}[mean and covariance dynamics of {naive} EKI]\label{lem:derivation_ODEs}
Consider the particle dynamics
\begin{equation*} 
     \dot{u}^j(t) = -C(t) A^T\Gamma^{-1}(Au^j(t)-y)
\end{equation*}
with sample covariance $C(t) :=\frac{1}{J}\sum_{j=1}^J(u^j(t)-m(t)) \otimes (u^j(t)-m(t))$ and sample mean $m(t) := \frac{1}{J}\sum_{j=1}^J u^j(t)$.
These two quantities are governed by the following differential equations:
\begin{alignat}{2}
    \dot{m}(t) &= -C(t) A^T\Gamma^{-1}(Am(t)-y), \quad &&m(0) = m_0,\\
    \dot{C}(t) &= -2C(t) A^T \Gamma^{-1} A C(t), \quad &&C(0) = C_0.
\end{alignat}
\begin{proof}
The differential equation for the mean follows directly by summing the particle equation.
The covariance dynamics can be derived as follows:
\begin{align*}
    \dot C(t) &= \frac{1}{J}\sum_{j=1}^J(\dot u^j(t) - \dot m(t))\otimes (u^j(t) - m(t)) + (u^j(t) - m(t))\otimes (\dot u^j(t) - \dot m(t)) \\
    &= -\frac{1}{J}\sum_{j=1}^J C(t)A^T\Gamma^{-1}A (u^j(t)-m(t))\otimes (u^j(t)-m(t)) \\
    &- \frac{1}{J}\sum_{j=1}^J (u^j(t)-m(t))\otimes \left[C(t)A^T\Gamma^{-1}A (u^j(t)-m(t)) \right]\\
    &= -C(t)A^T\Gamma^{-1}A\frac{1}{J}\sum_{j=1}^J  (u^j(t)-m(t))\otimes (u^j(t)-m(t))\\
    & - \frac{1}{J}\sum_{j=1}^J (u^j(t)-m(t))\otimes (u^j(t)-m(t)) \cdot A^T\Gamma^{-1}A C(t)\\
    &=-2C(t)A^T\Gamma^{-1}AC(t)
\end{align*}
\end{proof}
\end{lemma}

\begin{lemma}[mean and covariance dynamics of EKI]\label{lem:derivation_SDEs}
{%
We consider particles governed by 
\begin{align*}
     \dot{u}^j(t) = -C(t) A^T\Gamma^{-1}(Au^j(t)-y) + C(t)A^T\Gamma^{-1}\sqrt{\Sigma}\dot{\wiener}^j(t)
\end{align*}
with $\Sigma=\Gamma^{-1}$, sample covariance $C(t):=\frac{1}{J}\sum_{j=1}^J(u^j(t)-m(t)) \otimes (u^j(t)-m(t))$, and sample mean $m(t) := \frac{1}{J}\sum_{j=1}^J u^j(t)$.
These two quantities are governed by the following stochastic differential equations:
\begin{align}
    \dot m(t) &= -C(t)A^T\Gamma^{-1}(Am(t)-y) + C(t) A^T\Gamma^{-1/2}\dot{\overline\wiener}(t)\\
    \dot C(t) &= -\frac{J+1}{J}C(t) A^T\Gamma^{-1}A C(t) \\\notag
    &+ \frac1 {J}\sum_{j=1}^J e^j(t)\otimes (\dot\wiener^j(t)-\dot{\overline\wiener}(t)) \Gamma^{-1/2}A C(t) \\\notag
    &+  \frac1 {J}\sum_{j=1}^J C(t)A^T\Gamma^{-1/2}(\dot\wiener^j(t)-\dot{\overline\wiener}(t))\otimes e^j(t)
\end{align}
with $\overline\wiener(t) = \frac 1 J\sum_{j=1}^J\wiener^j(t)$.
In addition, the \emph{average {sample} mean} $\mathbf{m}(t):= \E^\wiener m(t)$ and \emph{average {sample} covariance} $\mathbf{C}(t) := \E^\wiener C(t)$ satisfy the following differential equations:
\begin{align}
    \dot{\mathbf C}(t) &= -\frac{J+1}{J}\mathbf C(t) A^T \Gamma^{-1} A \mathbf C(t)\notag
    \\
    &\qquad 
    -
    \E^\wiener\left[(C(t)-\mathbf C(t))A^T\Gamma^{-1}A(C(t)-\mathbf C(t))\right],\\
    \dot{\mathbf m}(t) &= -\mathbf C(t) A^T\Gamma^{-1}(A\mathbf m(t)-y) 
    \notag
    \\
    &\qquad 
    -\E^\wiener\left[(C(t)-\mathbf C(t))A^T\Gamma^{-1}A(m(t)-\mathbf m(t))\right]
\end{align}
\begin{proof}
The mean SDE is directly obtained by summing the particle dynamics. 
The SDE governing the covariance is computed as follows: By setting $e^j = u^j-m$ and using $C = \frac{1}{J}\sum_{j=1}^Je^j\otimes e^j$, It\=o's formula yields
\begin{align*}
    \de(e^j\otimes e^j) &= \de e^j\otimes e^j + e^j\otimes \de e^j + \frac{1}{2}\cdot 2\cdot \de e^j\otimes \de e^j\\
    &=-CA^T\Gamma^{-1}A(e^j\otimes e^j) \d t - (e^j\otimes e^j)A^T\Gamma^{-1}AC \d t\\
    &+CA^T\Gamma^{-1}(\de\wiener^j-\de\overline \wiener)\otimes e^j + e^j\otimes (\de\wiener^j-\de\overline \wiener)\Gamma^{-1/2}AC \\
    &+ CA^T\Gamma^{-1/2}(\de\wiener^j-\de\overline \wiener)\otimes (\de\wiener^j-\de\overline \wiener)\Gamma^{-1/2}AC.
\end{align*}
Now we use $(\de\wiener^j-\de\overline\wiener)\otimes (\de\wiener^j-\de\overline\wiener) = \frac{J-1}{J}E$, where $E$ is the identity matrix, which is a simple application of It\=o calculus, or a consequence of Lemma A.1 in \cite{blomker2019well}. Then we can sum over $j$ in order to obtain
\begin{align*}
    \de C(t) &= \frac{1}{J}\sum_{j=1}^J\de(e^j\otimes e^j)\\
    &=\left(-2 + \frac{J-1}{J}\right)CA^T\Gamma^{-1}AC\d t \\
    &+ \frac{1}{J}\sum_{j=1}^J CA^T\Gamma^{-1}(\de\wiener^j-\de\overline\wiener)\otimes e^j\\
    &+ \frac{1}{J}\sum_{j=1}^J e^j\otimes (\de\wiener^j-\de\overline\wiener) \Gamma^{-1}AC.
\end{align*}
We write this in integral form to obtain
\begin{align*}
    C(t) -C(0) &=\left(-2 + \frac{J-1}{J}\right)\int_0^t C(s)A^T\Gamma^{-1}AC(s)\d s \\
    &+ \frac{1}{J}\sum_{j=1}^J\int_0^t C(s)A^T\Gamma^{-1}(\de\wiener^j(s)-\de\overline\wiener(s))\otimes e^j\\
    &+ \frac{1}{J}\sum_{j=1}^Je^j\otimes\int_0^t (\de\wiener^j(s)-\de\overline\wiener(s)) \Gamma^{-1}AC(s).
\end{align*}
By proving that the stochastic integrals are indeed martingales (as in \cite{blomker2019well}), we can drop them after taking the expectation, i.e.
\begin{align*}
    \frac{\d}{\d t}\E C(t) &=\left(-2 + \frac{J-1}{J}\right) \E[C(t)A^T\Gamma^{-1}AC(t)] \\
    &= -\frac{J+1}{J}\E[C(t)]A^T\Gamma^{-1}A\E[C(t)] \\
    & -\frac{J+1}{J}\E[(C-\E[C])A^T\Gamma^{-1}A(C-\E[C])].
\end{align*}
The ODE for $\mathbf{m}(t)=\E^\wiener m(t)$ is derived analogously.
\end{proof}
}
\end{lemma}

{%

Finally, we provide a formal derivation of the mean-field dynamics.

\begin{lemma}[mean and covariance dynamics for mean-field EKI]\label{lem:derivation_meanfield}
Let $\rho(t,u)$ be a solution of
\begin{align*}
    \partial_t \rho = \div\left(\rho~\mathfrak C(t)A^T\Gamma^{-1}(A u - y)\right) + \frac12 \operatorname{Tr}(D^2\rho~\mathfrak C(t)A^T\Gamma^{-1}A\mathfrak C(t))
\end{align*}
and let $\mathfrak m(t)$ and $\mathfrak C(t)$ be defined by \labelcref{eq:meanfield_mean,eq:meanfield_cov}.
These two quantities are governed by the following differential equations:
\begin{align}
    \dot{\mathfrak{m}}(t) &=  -\mathfrak{C}(t)A^T\Gamma^{-1}(A\mathfrak{m}(t)-y),\\
    \dot{\mathfrak{C}}(t) &= -\mathfrak{C} (t) A^T\Gamma^{-1}A \mathfrak{C}(t).
\end{align}
\end{lemma}
\begin{proof}
The following computation is entirely formal, assuming that $\rho(t,u)$ is absolutely continuous with respect to the Lebesgue measure, and sufficiently smooth.

Using the definition of $\mathfrak m(t)$ and expanding the trace operator, one computes
\begin{align*}
    \dot{\mathfrak m}(t)
    &=\int u~\partial_t\rho(t,u)\d u\\
    &=\int u \div(\rho(t,u)\mathfrak C(t)A^T\Gamma^{-1}(Au-y))\d u \\
    &\qquad + \frac{1}{2}\sum_{i,j}\left(\mathfrak{C}(t)A^T\Gamma^{-1}A\mathfrak C(t)\right)_{ij}\int u \partial_i\partial_j \rho(t,u) \d u.
\end{align*}
Integrating the second term by parts, we see that it vanishes.
Integrating the first term by parts yields
\begin{align*}
    \dot{\mathfrak m}(t)
    &= - \mathfrak{C}(t)A^T\Gamma^{-1}A\int u~\rho(t,u)\d u + \mathfrak{C}(t)A^T\Gamma^{-1}y\int\rho(t,u)\d u \\
    &= - \mathfrak{C}(t)A^T\Gamma^{-1}(A\mathfrak{m}(t)-y),
\end{align*}
where we used $\int\rho(t,u)\d u=1$.

Using the dynamics of the mean we can derive the dynamics of the covariance as follows.
We compute using the product rule
\begin{align*}
    \dot{\mathfrak C}(t) &= 2\int (u-\mathfrak m(t))\otimes(\mathfrak C(t)A^T\Gamma^{-1}(A\mathfrak m(t)-y))\rho(t,u)\d u \\
    &\qquad+ \int (u-\mathfrak m(t))\otimes(u-\mathfrak m(t))\div(\rho(t,u)\mathfrak C(t)A^T\Gamma^{-1}(Au-y))\d u \\
    &\qquad+\mathfrak \sum_{i,j}\left(C(t)A^T\Gamma^{-1}A\mathfrak C(t)\right)_{ij}\int \frac{1}{2} (u-\mathfrak m(t))\otimes(u-\mathfrak m(t))\partial_i\partial_j\rho(t,u)\d u.
\end{align*}
Integrating the second term $(II)$ by parts shows
\begin{align*}
    (II)=-2\int (u-\mathfrak m(t))\otimes\left(\mathfrak C(t)A^T\Gamma^{-1}(Au-y)\right)\rho(t,u)\d u.
\end{align*}
Similarly, the third term $(III)$ integrates to
\begin{align*}
    (III)=\mathfrak C(t)A^T\Gamma^{-1}A\mathfrak C(t)\int \rho(t,u)\d u = \mathfrak C(t)A^T\Gamma^{-1}A\mathfrak C(t).
\end{align*}
Putting everything together one obtains
\begin{align*}
    \dot{\mathfrak C}(t)
    &=-2\int(u-\mathfrak m(t))\otimes(\mathfrak C(t)A^T\Gamma^{-1}A(u-\mathfrak m(t)))\rho(t,u)\d u + \mathfrak C(t)A^T\Gamma^{-1}A\mathfrak C(t)\\
    &=-2\int(u-\mathfrak m(t))\otimes(u-\mathfrak m(t))\rho(t,u)\d u \left(\mathfrak C(t)A^T\Gamma^{-1}A\right)^T + \mathfrak C(t)A^T\Gamma^{-1}A\mathfrak C(t)\\
    &=-\mathfrak C(t)A^T\Gamma^{-1}A\mathfrak C(t).
\end{align*}
\end{proof}
}
\end{appendix}
\end{document}

%% file: packages.tex
\usepackage[dvipsnames]{xcolor}
\usepackage[utf8]{inputenc}  
\usepackage[T1]{fontenc}
\usepackage{kantlipsum}
\usepackage{amsfonts}
\usepackage{tikz}
\usepackage{amsmath,amsthm,amssymb,upref}
\usepackage{todonotes}
\usepackage{comment}

\usepackage[normalem]{ulem}
\usepackage[margin=1.5in]{geometry}
\usepackage[framemethod=TikZ]{mdframed}

\usepackage[natbib,
maxbibnames=4,
sortcites
]{biblatex}
\usepackage{hyperref}
\usepackage[nameinlink, capitalise, noabbrev]{cleveref}
\usepackage{enumerate}

%% file: commands.tex
\numberwithin{equation}{section}

\newtheorem{lemma}{Lemma}[section]
\newtheorem{theorem}{Theorem}[section]
\newtheorem{proposition}{Proposition}[section]
\newtheorem{corollary}{Corollary}[section]

\theoremstyle{definition}

\theoremstyle{remark}
\newtheorem{remark}{Remark}[section]
\newtheorem{example}{Example}[section]

\newcommand{\de}{\mathrm d}

\newcommand{\N}{\mathbb N}
\newcommand{\R}{\mathbb R}

\newcommand{\E}{\mathbb E}
\newcommand{\eps}{\varepsilon}
\renewcommand{\d}{\,\mathrm{d}}
\newcommand{\st}{\,:\,}
\newcommand{\norm}[1]{\left\|#1\right\|}
\newcommand{\diag}{\operatorname{diag}}
\newcommand{\off}[1]{}
\DeclareMathOperator*{\argmin}{arg\,min}
\newcommand{\ran}{\mathrm{ran}}
\renewcommand{\div}{\mathrm{div}}

\newcommand{\cov}{\mathsf{C}}
\newcommand{\mean}{\mathsf{m}}
\newcommand{\preim}{\xi}
\newcommand{\wiener}{\mathbf{W}}

\usepackage{calc}